\documentclass{article}

\usepackage{import}

\usepackage{amsmath,amsfonts, amssymb, amsthm}
\usepackage{commath}
\usepackage{dsfont} 
\usepackage{bm} 
\usepackage{mathtools}

\usepackage{algorithm}
\usepackage{algorithmicx, setspace}
\usepackage{algpseudocode}
\algrenewcommand\algorithmicrequire{\textbf{Input:}}
\algrenewcommand\algorithmicensure{\textbf{Output:}}

\usepackage{graphicx}
\usepackage{subcaption}
\usepackage{tikz}
\usepackage{pgfplots}
\pgfplotsset{compat=1.18}
\usetikzlibrary{arrows.meta}	
\usetikzlibrary{backgrounds}
\usetikzlibrary{calc}	
\usetikzlibrary{patterns}
\usetikzlibrary{positioning}	
\usetikzlibrary{shapes}
\usepgfplotslibrary{fillbetween}
\usepackage{color}
\usepackage{wrapfig}

\usepackage{footmisc} 

\usepackage{hyperref}

\usepackage{fancyhdr}
\usepackage[top=1.2in, bottom=1in, left=1in, right=1in]{geometry}
\usepackage{xspace}
\usepackage{enumitem} 
\usepackage{parskip} 
\usepackage{booktabs}
\usepackage{tabularx}
\usepackage{tablefootnote}
\usepackage{multirow}
\usepackage{multicol}
\usepackage{afterpage}
\usepackage{moresize}

\usepackage{todonotes}

\usepackage{titlesec}
\titleformat{\chapter}[display]   
{\normalfont\huge\bfseries}{\chaptertitlename\ \thechapter}{20pt}{\Huge}   
\titlespacing*{\chapter}{0pt}{-30pt}{40pt}

\usepackage[titletoc,toc,page]{appendix}

\noappendicestocpagenum

\usepackage[natbib=true, citestyle=alphabetic, style=alphabetic, url=true, giveninits=true, uniquename=init]{biblatex}

\usepackage{color}
\usepackage{xcolor}
\usepackage{colortbl}
\definecolor{LightGreen}{rgb}{0.65,0.9,0.60}

\usepackage{thmtools}
\usepackage[capitalise]{cleveref}
\newtheorem{theorem}{Theorem}[section]

\newtheorem{lemma}[theorem]{Lemma}
\newtheorem{corollary}[theorem]{Corollary}

\theoremstyle{definition}
\newtheorem{definition}[theorem]{Definition}

\theoremstyle{remark}
\newtheorem{remark}[theorem]{Remark}

\newtheorem{theorem-informal}[theorem]{(Informal) Theorem}
\theoremstyle{plain}

\newtheorem*{notation*}{Notation}

\definecolor{LightBlue}{rgb}{0.6,0.6,1.}
\definecolor{VeryLightBlue}{rgb}{0.8,0.8,1.}
\definecolor{Salmon}{rgb}{1.,.57,.64}
\definecolor{Olive}{rgb}{0.3,0.4,0.1}
\definecolor{DarkRed}{rgb}{0.7,0,0}
\definecolor{DarkGreen}{rgb}{0.2,0.6,0.2}
\definecolor{DarkBlue}{rgb}{0,0,0.7}
\definecolor{LightCyan}{rgb}{0.88,1,1}
\definecolor{DarkCyan}{rgb}{0.68,0.8,0.8}

\definecolor{lightgrey}{rgb}{0.8,0.8,0.8}
\definecolor{lightergrey}{rgb}{0.9,0.9,0.9}
\definecolor{verylightgrey}{rgb}{0.95,0.95,0.95}

\renewcommand{\bar}{\widebar}
\makeatletter
\let\save@mathaccent\mathaccent
\newcommand*\if@single[3]{%
  \setbox0\hbox{${\mathaccent"0362{#1}}^H$}%
  \setbox2\hbox{${\mathaccent"0362{\kern0pt#1}}^H$}%
  \ifdim\ht0=\ht2 #3\else #2\fi
  }
\newcommand*\rel@kern[1]{\kern#1\dimexpr\macc@kerna}
\newcommand*\widebar[1]{\@ifnextchar^{{\wide@bar{#1}{0}}}{\wide@bar{#1}{1}}}
\newcommand*\wide@bar[2]{\if@single{#1}{\wide@bar@{#1}{#2}{1}}{\wide@bar@{#1}{#2}{2}}}
\newcommand*\wide@bar@[3]{%
  \begingroup
  \def\mathaccent##1##2{%
    \let\mathaccent\save@mathaccent
    \if#32 \let\macc@nucleus\first@char \fi
    \setbox\z@\hbox{$\macc@style{\macc@nucleus}_{}$}%
    \setbox\tw@\hbox{$\macc@style{\macc@nucleus}{}_{}$}%
    \dimen@\wd\tw@
    \advance\dimen@-\wd\z@
    \divide\dimen@ 3
    \@tempdima\wd\tw@
    \advance\@tempdima-\scriptspace
    \divide\@tempdima 10
    \advance\dimen@-\@tempdima
    \ifdim\dimen@>\z@ \dimen@0pt\fi
    \rel@kern{0.6}\kern-\dimen@
    \if#31
      \overline{\rel@kern{-0.6}\kern\dimen@\macc@nucleus\rel@kern{0.4}\kern\dimen@}%
      \advance\dimen@0.4\dimexpr\macc@kerna
      \let\final@kern#2%
      \ifdim\dimen@<\z@ \let\final@kern1\fi
      \if\final@kern1 \kern-\dimen@\fi
    \else
      \overline{\rel@kern{-0.6}\kern\dimen@#1}%
    \fi
  }%
  \macc@depth\@ne
  \let\math@bgroup\@empty \let\math@egroup\macc@set@skewchar
  \mathsurround\z@ \frozen@everymath{\mathgroup\macc@group\relax}%
  \macc@set@skewchar\relax
  \let\mathaccentV\macc@nested@a
  \if#31
    \macc@nested@a\relax111{#1}%
  \else
    \def\gobble@till@marker##1\endmarker{}%
    \futurelet\first@char\gobble@till@marker#1\endmarker
    \ifcat\noexpand\first@char A\else
      \def\first@char{}%
    \fi
    \macc@nested@a\relax111{\first@char}%
  \fi
  \endgroup
}
\makeatother
\makeatletter
\renewcommand{\P}{%
	\@ifnextchar\bgroup%
	{\@Pwithargs}
	{\@Pnoargs}
}
\newcommand{\@Pwithargs}[1]{%
	\@ifnextchar\bgroup%
	{\@Ptwoargs{#1}}
	{\@Ponearg{#1}}
}
\newcommand{\@Pnoargs}{\mathbb{P}}
\newcommand{\@Ponearg}[1]{\mathbb{P}\left[ #1 \right]}
\newcommand{\@Ptwoargs}[2]{\mathbb{P}_{#1}\left[ #2 \right]}
\newcommand{\E}{%
	\@ifnextchar\bgroup%
	{\@Ewithargs}
	{\@Enoargs}
}
\newcommand{\@Ewithargs}[1]{%
	\@ifnextchar\bgroup%
	{\@Etwoargs{#1}}
	{\@Eonearg{#1}}
}
\newcommand{\@Enoargs}{\mathbb{E}}
\newcommand{\@Eonearg}[1]{\mathbb{E}\!\left[ #1 \right]}
\newcommand{\@Etwoargs}[2]{\underset{#1}{\mathbb{E}}\!\left[ #2 \right]}
\makeatother

\newcommand{\var}{\mathrm{var}}
\newcommand{\normal}{\mathcal{N}}
\newcommand{\normalof}[1]{\normal \! \left( #1 \right) }

\newcommand{\eqd}{\overset{d}{=}}

\newcommand{\reals}{\mathbb{R}}
\newcommand{\naturals}{\mathbb{N}}

\newcommand{\vv}{\bm{v}}

\newcommand{\vx}{\bm{x}}

\newcommand{\vSigma}{\boldsymbol{\Sigma}}

\renewcommand{\P}{\mathbb{P}}

\newcommand{\R}{\mathbb{R}}

\newcommand{\DD}{\mathcal{D}}

\newcommand{\vzero}{\mathbf{0}}

\newcommand{\eye}{\mathbf{I}}

\newcommand{\indicator}{\mathds{1}}

\mathchardef\mhyphen="2D

\DeclareMathSymbol{\shortcol}{\mathord}{operators}{"3A}

\newcommand{\bigO}[1]{\mathcal{O}(#1)}

\newcommand{\adj}{\Gamma}

\newcommand{\restr}[2]{{\left.\kern-\nulldelimiterspace #1 \right|_{#2}}}

\newcommand{\vbt}{\texttt{VirtualBrownianTree} }
\newcommand{\vbteval}{\texttt{VirtualBrownianTree.eval}}

\newcommand{\bridge}{\texttt{bridge}}
\newcommand{\finalinterp}{\texttt{final\_interpolation}}
\newcommand{\splitseed}{\texttt{split\_seed}}

\newcommand{\topV}[1]{{\left\lceil #1 \right\rceil}_V}
\newcommand{\botV}[1]{{\left\lfloor #1 \right\rfloor}_V}

\newcommand{\barH}{\bar{H}}
\newcommand{\barK}{\bar{K}}
\newcommand{\barY}{\bar{Y}}

\title{Single-seed generation of Brownian paths and integrals\\ for adaptive and high order SDE solvers}
\author{%
  \textbf{Andra\v{z} Jelin\v{c}i\v{c}}
  \quad\quad\quad
  \textbf{James Foster}\\
  Department of Mathematical Sciences\\
  University of Bath, UK \\
  \texttt{\{aj2382, jmf68\}@bath.ac.uk}\\
  \and
  \textbf{Patrick Kidger}\\
  Cradle.bio\\
  Zurich, Switzerland\\
  \texttt{math@kidger.site}
}

\hypersetup{
  pdftitle={Single-seed generation of Brownian motion for SDE simulation},
  pdfauthor={Andraž Jelinčič, James Foster and Patrick Kidger}
}

\begin{document}

\maketitle

\vspace{2cm}

\begin{abstract}
    Despite the success of adaptive time-stepping in ODE simulation, it has so far seen few applications for Stochastic Differential Equations (SDEs). To simulate SDEs adaptively, methods such as the Virtual Brownian Tree (VBT) have been developed, which can generate Brownian motion (BM) non-chronologically. However, in most applications, knowing only the values of Brownian motion is not enough to achieve a high order of convergence; for that, we must compute time-integrals of BM such as $\int_s^t W_r \, dr$. With the aim of using high order SDE solvers adaptively, we extend the VBT to generate these integrals of BM in addition to the Brownian increments. A JAX-based implementation of our construction is included in the popular Diffrax library (\textcolor{blue}{\href{https://github.com/patrick-kidger/diffrax}{github.com/patrick-kidger/diffrax}}).
    
    Since the entire Brownian path produced by VBT is uniquely determined by a single PRNG seed, previously generated samples need not be stored, which results in a constant memory footprint and enables experiment repeatability and strong error estimation. Based on binary search, the VBT's time complexity is logarithmic in the tolerance parameter $\varepsilon$. Unlike the original VBT algorithm, which was only precise at some dyadic times, we prove that our construction exactly matches the joint distribution of the Brownian motion and its time integrals at any query times, provided they are at least $\varepsilon$ apart.

    We present two applications of adaptive high order solvers enabled by our new VBT. Using adaptive solvers to simulate a high-volatility CIR model, we achieve more than twice the convergence order of constant stepping. We apply an adaptive third order underdamped (or kinetic) Langevin solver to an MCMC problem, where our approach outperforms the No U-Turn Sampler, while using only a tenth of its function evaluations.
\end{abstract}



\clearpage

\section{Introduction}
In this paper, we introduce a method for simulating Brownian motion (BM) which can power both adaptive time-stepping and high order numerical solvers for Stochastic Differential Equations (SDEs). The former needs the ability to query the Brownian path non-chronologically, and the latter requires access to integrals of BM, also known as ``Lévy areas" (see \cref{def:intro:levy_bb}). The capability to combine these is new here.

Rössler \citep{rossler2010runge} proposed a solver for additive-noise SDEs which achieves strong order $3/2$ (see \cref{def:intro:strong_order}), provided access to space-time Lévy area (see \cref{def:intro:levy_bb}). In \cite{scott2025quicsort} the authors found a solver for the underdamped/kinetic Langevin diffusion with strong order 3, provided access to space-time-time Lévy area (\cref{def:intro:levy_bb}). The aim of this paper is to enable these high order solvers to be used adaptively, which requires a new method of simulating Brownian motion and its Lévy areas.

Although adaptive solvers such as Dormand--Prince 5(4) \cite{DoPri1980} are widely used in ODE simulation \citep{burrage2004adaptive}, applying them to SDEs \citep{gaines97variable,ilie2015adaptive} is more challenging and far less common. To see why, consider the usual way to generate a BM $W$ by drawing independent Gaussian random variables $N(0, h)$ for each increment $W_{t + h} - W_t$. When an adaptive solver's internal error estimate is too large, it backtracks and redoes the last step but with a smaller step size $h' < h$. Since the error estimate depends on the value of $W_{t + h} - W_t$, it cannot be discarded, but the new sample $W_{t + h'}$ must be conditioned on $W_{t + h}$. This backtracking sometimes needs to be repeated several times. To facilitate this, non-chronological BM generators have been developed \citep{Rackauckas2017adaptive, kidger2021efficient}. One of these is the Virtual Brownian Tree (VBT) \citep{li2020scalable}, on which our method is based. 

Most methods of generating Brownian motion are query-dependent, meaning that the generated Brownian path depends both on the random seed, and the query times $s,t\in \R$ at which $W_t - W_s$ is evaluated. With the VBT, however, the initial random seed completely determines the entire Brownian path. This has the downside of requiring a pre-specified tolerance parameter $\varepsilon$ and having an $\bigO{\log \varepsilon}$ time complexity per evaluation (see \cref{sec:vbt}). However, it brings several benefits, including an $\bigO{1}$ memory cost, easier experiment repeatability, straightforward computation of solver orders, and a concise proof of correctness.

To use the aforementioned SDE solvers adaptively, we extend the VBT to additionally generate the space-time and space-time-time Lévy areas of BM. Like the VBT, our approach is query-independent and has the same computational complexity. In Theorems \ref{thm:theory:arbitrary_whk} and \ref{prop:impl:match_dist} we show that our construction exactly matches the joint distribution of the Brownian increments and Lévy areas for all query points separated by at least $\varepsilon$. For contrast, the original VBT was precise only at dyadic query times and only generated Brownian increments. Our implementation is available as part of the software package Diffrax (\textcolor{blue}{\href{https://github.com/patrick-kidger/diffrax}{github.com/patrick-kidger/diffrax}}).

\subsection{Structure of this paper}

\cref{sec:back} introduces key terminology and related work, with a detailed explanation of VBT in \cref{sec:vbt}. 

In \cref{sec:theory:interp_bm} we present our new mathematically-precise interpolation method for the Brownian-increments-only case of VBT and prove that it produces correctly distributed samples also at non-dyadic times. The rest of \cref{sec:theory} contains the theoretical results required to generate the space-time L\'{e}vy area $H$ and space-time-time L\'{e}vy area $K$ adaptively. \cref{sec:theory:chen} gives a new version of Chen's relation for the triple $(W, H, K)$. \cref{sec:theory:midpoint} gives the rule for generating these at dyadic points which is then extended to non-dyadic times in \cref{sec:theory:general_times}.

In \cref{sec:Implementation}, we condense this theory into an improved algorithm for the VBT. In addition to augmenting the VBT with  L\'{e}vy areas, we introduce a further improvement called ``Interval Normalisation" (see \cref{sec:IntervalNormalisation}), which substantially reduces numerical errors. In \cref{sec:impl:true_distn} we prove that given a mild assumption on the spacing of query points, our construction exactly matches the full pathwise joint distribution of the BM and its L\'{e}vy areas. 

In \cref{sec:application} we present two applications of the new VBT algorithm, first to the Cox-Ingersoll-Ross model, which is provably difficult to solve with constant step sizes, and second to MCMC sampling problems, where we demonstrate the performance of the third order Langevin solver when used adaptively.
\clearpage

\section{Related Work and Theoretical Background}
\label{sec:back}

\paragraph{SDE simulation} Suppose the task is to simulate an SDE of the form 
\[
dX_t = f(X_t)dt + \sum_{i=1}^d g_i(X_t)dW_t^{(i)},\ \ X_0=x_0,
\]
where the solution $X = \left( X_t \right)_{t\in[0,T]}$ takes values in $\R^e$, $W = (W^{(1)},\dots,W^{(d)})$ denotes a standard $d$-dimensional Brownian motion (BM) and $f, g_i : \R^e \rightarrow \R^e$ are suitably regular vector fields. SDE simulation comes in two flavors: weak and strong. In the former, we are interested in distributional properties of the random variable $X_t$, such as its moments; whereas in the case of strong simulation, we are given a particular path of the BM $w: [0, T] \rightarrow \R^d$ and the aim is to generate a sample path $\hat{X}(w) = \big( \hat{X}_t(w) \big)_{t\in[0,T]}$ which satisfies the SDE together with $w$. An important measure of success for an SDE solver is its strong order of convergence (SOC).
\begin{definition}[Strong order of convergence (SOC)]
\label{def:intro:strong_order}
    Suppose we have a strong SDE solver, which given a Brownian path $w$ produces a solution $\hat{X}_N(w)$ using $N$ computational steps. Let $X(w)$ be a true solution of the SDE corresponding to $w$ (usually obtained via a numerical solver with very small steps). Then the strong error of the solver is
    \[
        \epsilon_{\text{strong}}(N) \coloneqq \sup_{0 \leq n \leq N} \, {\E{w}{ \left| {\hat{X}_N(w)}_{t_n} - {X(w)}_{t_n} \right|^2}}^{\frac{1}{2}}.
    \]
    The solver has a SOC $\gamma$ if there exists a constant $C > 0$ s.t. $ \epsilon_{\text{strong}}(N) \leq C \left( \frac{T}{N} \right) ^{\gamma} \;\; \forall \, N \in \naturals$. When using a constant-step solver $h = \frac{T}{N}$ is just the step size, whereas in adaptive solvers it is the mean step size.
    
\end{definition}
A more in-depth introduction to SDE simulation can be found in \cite{higham2021introduction}.

\begin{figure}[hbp]
    \centering
    \includegraphics[width=0.8\textwidth]{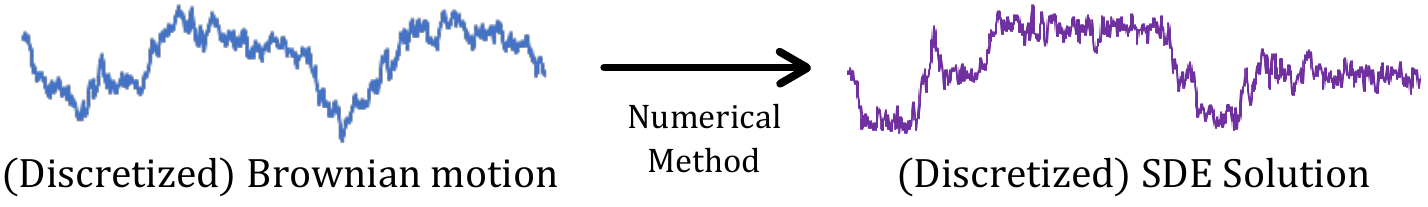}
    \caption{An SDE solver is a deterministic function from a Brownian path to an SDE trajectory \citep{foster2023convergence}.}
    \label{fig:strong_sde_solving}
\end{figure}

Depending on the SDE being simulated, a solver may need access to certain ``Levy areas" in order to achieve a high SOC.

\begin{restatable}[Brownian bridge and Lévy areas]{definition}{LvyBBdef}
\label{def:intro:levy_bb}
Let $0 \leq s \leq r \leq t$ and let $W: [0,\infty) \rightarrow \reals^d$ be a Brownian motion. Write $W_{s,t} = W_t - W_s$. Then the Brownian bridge on $[s,t]$ evaluated at $r$ is given by
\begin{equation*}
    B_r^{s,t} \coloneqq W_r \, | \, \{W_s = W_t = 0\} = W_{s,r} - \frac{r-s}{t-s} W_{s,t} \quad \text{ with convention that when } s=t, \;\; B_s^{s,s} \coloneqq 0. 
\end{equation*}
We define the ``space-space Lévy area" $A_{s,t} \in \reals^{d\times d}$ (usually called just ``Lévy area") as
\[
    A_{s,t}^{(i,j)} \coloneqq \frac{1}{2} \left(\int_{s}^{t} W_{s,r}^{(i)} \, dW_r^{(j)}-\int_{s}^{t} W_{s,r}^{(j)} \, dW_r^{(i)}\right), \quad \text{ for } \; i, j \in \set{1, \ldots, d},
\]
\begin{align*}
    &\text{the ``space-time Lévy area" $H_{s,t} \in \reals^d$ as} \quad && H_{s,t} \coloneqq \frac{1}{t-s} \int_s^t B_r^{s,t} dr, \; \text{ and} \\[1mm]
    &\text{and the ``space-time-time Lévy area" $K_{s,t} \in \reals^d$ as} \quad  && K_{s,t} \coloneqq \frac{1}{(t-s)^2} \int_s^t B_r^{s,t} \left( \frac{t+s}{2} - r \right) dr.
\end{align*}
\end{restatable}

For general multidimensional SDEs, a solver without access to space-space Lévy area $A$ can attain a SOC of at most $\gamma = \frac{1}{2}$ \citep{ClarkCameron}.
Unfortunately, space-space Lévy area is not Gaussian and notoriously difficult to sample from. Approximations to L\'evy area have been well studied (see \cite{Davie,Dickinson,jelincic2023generative,KPW,foster2020a,FH,MR,Wiktorsson}), and the only method of generating space-space Lévy area with SOC higher than $\frac{1}{2}$ is that by \cite{GL94levy}, which only applies when $d=2$.

\paragraph{High order solvers with space-time Lévy area} For certain classes of SDEs, a high SOC can be achieved even with access to only space-time Lévy area $H$, which is Gaussian and hence easier to sample than space-space Lévy area $A$. In particular \cite{rossler2010runge} proposes a solver using space-time Lévy area with a SOC of $1.5$ when applied to additive-noise SDEs, that is when $g$ has no dependence on $X_t$ and can be written as just $g(t)$. \cite{foster2023high} extended this to a more general class of SDEs called commutative noise SDEs, in which the diffusion vector field $g$ commutes in the Lie bracket (i.e.~the columns satisfy $g_i' g_j = g_i g_j'$).

Another class of SDEs that admit high order solvers are the Langevin diffusions, which originate from molecular dynamics \citep{langevin97}, but are now commonly used in Markov Chain Monte Carlo applications \citep{brooks2011mcmc,li2019stochMCMC}. While access to space-time Lévy area enables solvers to achieve SOC 2, it is possible to achieve even SOC 3 if the solver also uses space-time-time Lévy area (which is also Gaussian and easy to generate) \citep{scott2025quicsort}.

\paragraph{Adaptive time-stepping} is a widely used technique in ODE simulation \citep{soderlind2002adaptiveODE,hairer93ODEsBook}. It involves numerically estimating the error of the solver at every step of the solution, and decreasing the step size when errors are too large, or increasing when they are small. However, for SDEs, constant time-stepping is still the predominant paradigm, despite the fact that some problems, such as simulating the CIR model, provably cannot be solved efficiently with constant steps \citep{hefter2019slowCIR}.

The first significant attempt at adaptive stepping for SDE simulation is by \cite{gaines97variable}, where steps of the solver can be halved if large numerical errors are detected. To enable this, the authors proposed an early iteration of the Brownian Tree, which only supported dyadic time-step refinements and required keeping the results of previous queries in memory.

More advanced step-size controllers, such as the PID controller \citep{soderlind2003digital,hairer2002solving-ii,hairer2008solving-i}, can also be used for SDE simulation \citep{ilie2015adaptive}. However, to facilitate that, the BM generation method must also accept queries at non-dyadic times  (i.e.~instead of always
just halving or doubling the step, we permit an arbitrary step size), so methods such as the Brownian Interval by \cite{kidger2021efficient} and Rejection Sampling with Memory (RSwM) by \cite{Rackauckas2017adaptive} were proposed. These are unfortunately query dependent, and require storing previously drawn samples in memory. To address this issue \cite{li2020scalable} proposed the Virtual Brownian Tree, which is query-independent and has a constant memory cost, but is imprecise at non-dyadic points, due to its use of piecewise-constant interpolation. Our new interpolation method proposed in \cref{sec:theory:interp_bm} fixes this and matches the target distribution exactly.

\paragraph{Lévy's construction} The VBT is inspired by Lévy's construction of Brownian motion, which splits the interval $[0,1]$ into progressively smaller dyadic sub-intervals. This means that at level $d$, we have evaluated the BM at all $t$ of the form $k2^{-d}$ with $ \; k \in \{0, 1, ..., 2^d \}.$ As we descend through the levels, this path converges to BM. The way to do this for Brownian increments is well known, and has long been used as a neat proof for the existence of BM. Adding a version of space-time Lévy area to this construction was first discussed by \cite{mauthner1998stepsize,burrage2004adaptive}.
\cite{foster2020a} streamlined this construction and found a way to generate space-time-time Lévy area at dyadic points. Here we extend the latter to non-dyadic points, integrate all this into the VBT algorithm and prove the correctness of its output distribution.

\subsection{The Virtual Brownian Tree}\index{Brownian!Tree}
\label{sec:vbt}

The Virtual Brownian Tree (VBT), proposed by \cite{li2020scalable} and inspired by \cite{gaines97variable}, represents a deterministic function $f_{\text{VBT}}$ from the space of $m$-bit seeds $\set{0, 1}^m$ to the space of continuous paths $\mathcal{C} \! \left([t_0, t_1], \reals^d \right)$. This means that only a single seed needs to be stored to reconstruct the entire path, which reduces the memory cost to $\bigO{1}$. Unlike other approaches (Brownian Interval by \cite{kidger2021efficient} or RSwM by \cite{Rackauckas2017adaptive}), which have to store previous queries, the VBT is stateless. While most methods are concerned with the distribution of each individual sample conditioned on previously generated samples, VBT is concerned with the distribution of the entire path $w \colon [t_0, t_1] \to \reals^d$. Specifically, if $\mathcal{U}$ is a uniform measure in seed space, then we want the push-forward $\mathcal{U} \circ f^{-1}_{\text{VBT}}$ to be close to the Wiener measure on the space of paths (which can be matched exactly on sufficiently spaced cylinder sets using our new interpolation method proposed in \cref{sec:theory:interp_bm}). 

\begin{remark}
\label{remark:pseudo_random_normal}
    Note that a continuous random variable cannot be precisely matched by its floating-point computational approximation, which necessarily introduces some discretisation. Here and throughout the paper we neglect this discrepancy.
\end{remark}

\begin{figure}[htp]
    \centering
    \includegraphics[width=0.5\textwidth]{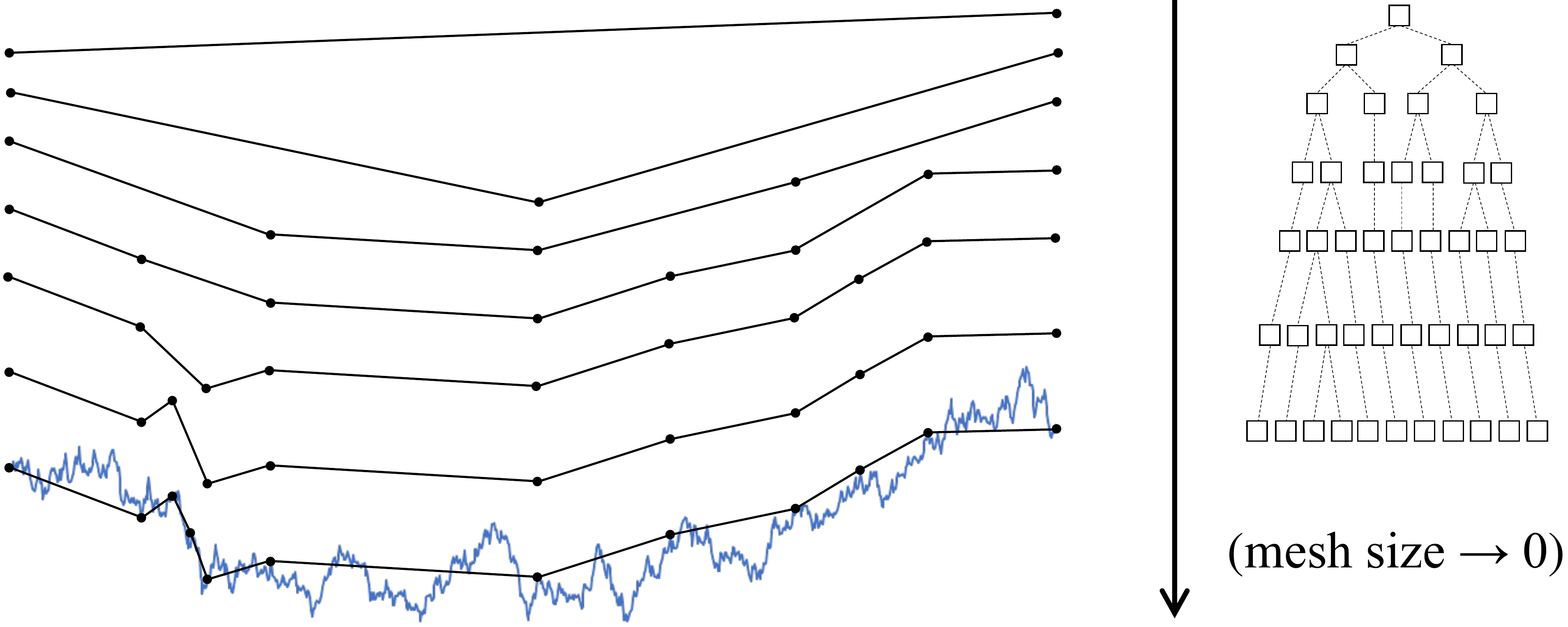}
    \caption{An illustration of the original Brownian Tree by \cite{gaines97variable} (image from \cite{foster2023convergence}).}
    \label{fig:BT_illustration}
\end{figure}

\subsubsection{Splittable PRNGs}\label{sec:vbt:splittable-prng}
The first key ingredient is ``splittable" pseudo-random number generator (PRNG) seeds \citep{prng1, prng2}.

Given an $m$-bit random seed $\rho \in \{0, 1\}^m$, splitting is an operation that produces some $n$ new $m$-bit random seeds $\rho_1, \ldots, \rho_n \in \{0, 1\}^m$, as a deterministic function of $\rho$, for which $\rho, \rho_1, \ldots, \rho_n$ produce statistically independent streams of random numbers when used as the seed for a PRNG.

Given any rooted tree, we can associate a random seed with every node in the tree in the following way.

Let $(V, E, *)$ be a rooted tree, where $V$ is a vertex set,
\begin{equation*}
E \subseteq \set{ \{x, y\} \, : \, x, y \in V, \; x \neq y}
\end{equation*}
a corresponding undirected edge set (connected and without cycles), and $* \in V$ denotes the root. For any $x \in V$, let $\adj(x) = \set{y \in V : \{x, y\} \in E}$ denote the set of vertices adjacent to $x$.

Let $\rho \in \{0, 1\}^m$ be an $m$-bit seed, which we associate with the root $*$. Split $\rho$ into $\rho_1, \ldots, \rho_{|\adj(*)|}$ random seeds and pair each one with a corresponding element $v_i \in \adj(*)$. Recursively split each $\rho_i$ and pair the resulting seeds with the elements of $\adj(v_i)$, etc., repeating this procedure throughout the tree.

By fixing a rooted tree $(V, E, *)$ and a root seed $\rho$, we may deterministically create a PRNG at every node in the tree. Provided we remember only the tree structure and the root seed $s$, we can later rematerialise every PRNG sequence, for every node of the tree, without holding the samples in memory.

\subsubsection{Generating Brownian samples}
\label{sec:vbt:generating_samples}
We generate the value of the Brownian path $w: [t_0, t_1] \rightarrow \reals^d$ at a time $r \in [t_0, t_1]$ using a binary-search-like procedure. We start with $s=t_0,$ $u=t_1$. In each step we set $t = \frac{s+u}{2}$ and use the values of $w_s$ and $w_u$ to compute $w_t$ using the Brownian bridge function (a special case of \cref{eq:vbt:bridge-final-interp}):
\begin{equation}
    \label{eq:vbt:midpoint_bridge}
    w_t = \texttt{bridge}(s, t, u, w_s, w_u, \widehat{\rho} \,) \coloneqq \frac{w_s + w_u}{2} + \tfrac{1}{2}\sqrt{u-s} \, \normalof{0, \eye_d, \widehat{\rho} \,}.
\end{equation}
Here, $\normalof{\cdot, \cdot, \widehat{\rho} \,}$ denotes a Gaussian pseudo-random variable generated using the seed $\widehat{\rho}$, which is the PRNG seed corresponding to the current node in the binary search tree (as per \cref{sec:vbt:splittable-prng}).
To descend into a child node, we halve the interval by setting either $s \gets t$ or $u \gets t$ such that the query time $r$ remains in $[s, u]$. This is repeated until the size of $[s, u]$ is smaller than the desired tolerance $\varepsilon > 0$.

Let $L = \left\lceil - \log_2 (\varepsilon) \right\rceil$ and $V = \set{t_0 + (t_1 - t_0) k 2^{-L} : k \in \set{0, 1, \ldots , 2^L} }$. By recording only the root-level seed $\sigma \in \{0, 1\}^m$, and computing the seeds of child nodes as we descend the tree, the Brownian sample $w_v$ is completely determined for all $v \in V$. The full procedure is given in \cref{alg:vbt:basic-alg}.

\begin{algorithm}[H]
\caption{Sampling from the Virtual Brownian Tree using the $\vbteval$ function. \texttt{split\_seed} denotes splitting a PRNG seed, and \texttt{bridge} denotes the Brownian bridge from \cref{eq:vbt:midpoint_bridge}.}\label{alg:vbt:basic-alg}

\hspace*{2.5mm} \textbf{Input:} $r$ - evaluation time, $[t_0, t_1]$ - definition interval, $\rho$ - PRNG seed, $d$ - dimension, $\varepsilon$ - tolerance.
\vspace{1mm}

\hspace*{2.5mm} \textbf{Result:} Approximation to $w_r$
\vspace{1mm}

\begin{algorithmic}
\State $\rho, \widehat{\rho} \gets \texttt{split\_seed}(\rho, 2)$

\State $s \gets t_0, \;\; u \gets t_1$
\State $w_s \gets 0$
\State $w_u \gets \normalof{0, \, (t_1 - t_0)\eye_{d_w}, \widehat{\rho}\,}$
\vspace{2mm}
\While{$\abs{u-s} > \varepsilon$}
    \State $t \gets (s + u) / 2$
    \State $\rho_1, \rho_2, \widehat{\rho} \gets \texttt{split\_seed}(\rho, 3)$  \Comment{$\widehat{\rho}$ to generate the midpoint, and $\rho_1, \rho_2$ for children nodes.}
    \State $w_t \gets \texttt{bridge}(s, t, u, w_s, w_u, \widehat{\rho}\,)$
    \If{$r > t$}  \Comment{Descend into the \textbf{right} child.}
        \State $s \gets t,$ $\; w_s \gets w_t$
        \State $\rho \gets \rho_1$
    \Else  \Comment{Descend into the \textbf{left} child.}
        \State $u \gets t,$ $\; w_u \gets w_t$
        \State $\rho \gets \rho_2$
    \EndIf
\EndWhile

\State $w_r \gets \finalinterp \! \left(s, r, u, w_s, w_u, \rho \right) $ \Comment{Defined below, originally piecewise constant.}
\State \Return $w_r$
\end{algorithmic}
\end{algorithm}

The main use case will be to sample a Brownian increment $w(r_1) - w(r_0)$ (when an SDE solver steps from $r_0$ to $r_1$). In that case, we separately sample $w(r_0)$ and $w(r_1)$ using \cref{alg:vbt:basic-alg}, and then return $w(r_1) - w(r_0)$. While this takes $\bigO{\log(1/\varepsilon)}$ time, it has the advantage that it requires only $\bigO{1}$ memory.

In the original algorithm, $\finalinterp$ was just piecewise constant, but in \cref{sec:theory:interp_bm} we propose an alternative that guarantees the correct distribution at all times, not only at dyadic ones.

\subsubsection{The VBT and strong SDE simulation}
As long as we use the same seed, the VBT will always generate the same Brownian path, and so all solvers using that Brownian path as input will always follow the same trajectory of the simulated SDE, no matter what step sizes they use. With the usual query-dependent BM generators (where query times can influence the resulting path), this would only be the case if all solvers used exactly the same step sizes.

The usual approach for computing the strong order of an SDE solver is to first generate a very finely discretised Brownian path, store it in memory, and then run the solver using that path. When using GPU acceleration libraries to alleviate the computational cost of strong SDE simulation (e.g.~in the case of diffusion models, \cite{song2021scorebased}) the above workflow needs to be parallelised, and hundreds of finely discretised Brownian paths must be stored in video memory simultaneously. By using the VBT one can ensure that all of the solvers being compared will be fed the same Brownian path, no matter at what times they query it, while needing to store just a single 64-bit seed per path. This is especially relevant when evaluating adaptive solvers, for which we cannot pre-specify at which times they will evaluate the Brownian path.

 Neural SDEs \citep{chen18NODEs,kidger2021efficient,kidger2021neural} pose a similar challenge. To train these, the Brownian path used on the forward pass must be evaluated backwards in time. Normally, this requires storing all samples in memory, incuring a significant cost. In addition, when training via ``optimise-then-discretise" the backwards pass might query the BM at different times than the forward pass, while requiring the new samples to be conditioned on the samples of the forward pass. Both of these difficulties can be avoided by using the VBT.

\section{Exact simulation of Brownian motion and L\'{e}vy area}
\label{sec:theory}

\subsection{New interpolation method for Brownian increments}
\label{sec:theory:interp_bm}
Let $\varepsilon$ be the tolerance parameter, $L = \left\lceil - \log_2 (\varepsilon) \right\rceil$ and $V = \set{t_0 + (t_1 - t_0) k 2^{-L} : k \in \set{0, 1, \ldots , 2^L} }$. Suppose $w: [t_0, t_1] \rightarrow \reals^d$ given by $r \mapsto \vbteval(r, [t_0, t_1], \rho, d, \varepsilon)$ is the path generated by \cref{alg:vbt:basic-alg}. By L\'{e}vy's construction of Brownian motion, we have that if the seed $\rho$ was drawn uniformly at random and $W$ is a Brownian motion, then (with $\eqd$ denoting equality in distribution):
\begin{equation*}
    \left( w_v \right)_{v \in V} \eqd \left( W_v \right)_{v \in V}.
\end{equation*}
However, the original VBT algorithm just snaps the query time $r$ to the nearest element of $V$, so its output distribution is imprecise at all points $r \in [t_0, t_1] \setminus V$. We can fix this by using an interpolation based on the distribution of the Brownian bridge:
\begin{equation}
\label{eq:vbt:bridge-final-interp}
\begin{split}
    &\text{for} \; s<r<u \;\; \left(W_r | W_s, W_u \right) \sim \normalof{W_s + \frac{r-s}{u-s} W_{s,u}, \frac{(r-s)(u-r)}{u-s}}, \;\; \text{so we set} \\[1mm]
    &\finalinterp \! \left(s, r, u, w_s, w_u, \rho \right) = w_s + \frac{r-s}{u-s} \, (w_u - w_s) + \sqrt{\frac{(r-s)(u-r)}{u-s}} \: \normalof{0, \eye_d, \rho}
\end{split}
\end{equation}
The following proposition states that our generated path now exactly matches the desired distribution, as long as we draw at most one sample between each two consecutive points in $V$.

\begin{theorem}[Distributional correctness of VBT]
\label{prop:vbt:match_dist}
    Let $\rho$ be a PRNG seed, let $t_0 < t_1 \in \reals$, $L, d \in \naturals$, $\varepsilon = (t_1 - t_0) 2^{-L}$ and
    \[
    w_r \coloneqq \vbteval \! \left(r, [t_0, t_1], \rho, d, \varepsilon \right),
    \]
    where the \vbt is implemented using the $\finalinterp$ from \cref{eq:vbt:bridge-final-interp}.

    Let $V = \set{t_0 + (t_1 - t_0) k 2^{-L} : k \in \set{0, 1, \ldots , 2^L} }$ be the set of vertices of the VBT and let $S = \set{r_1, \ldots , r_N} \subset [t_0, t_1]$ be a set of query times such that for all $0 \leq n < N$ we have $r_n \leq r_{n+1}$ and $ \left\vert V \cap [r_n, r_{n+1}] \right\vert \geq 1$.

    If $W$ is a Brownian motion and $\rho$ was chosen uniformly at random (but is fixed throughout), then the joint distribution of $w$ sampled at points from $S$ matches the joint distribution of $W$ sampled at the same points:
    \begin{equation*}
        \left( w_{r_1}, \ldots, w_{r_N} \right) \eqd \left( W_{r_1}, \ldots, W_{r_N} \right).
    \end{equation*}
\end{theorem}
\begin{proof}
    This is a special case of \cref{prop:impl:match_dist}, in which we extend this result to include L\'{e}vy areas as well.
\end{proof}

\subsection{Chen's relation for L\'{e}vy area}
\label{sec:theory:chen}

Recall the space-time L\'{e}vy area $H$ and the space-time-time L\'{e}vy area $K$ from \cref{def:intro:levy_bb}.

In order for the \vbt to generate these L\'{e}vy areas, we must devise versions of the $\texttt{bridge}$ and $\finalinterp$ functions that work not only for the Brownian motion $W$, but for the pair $(W, H)$ or the triple $(W, H, K)$. The first step is finding a way of concatenating L\'{e}vy areas over consecutive intervals.

For times $0 \leq s \leq t \leq u$, Brownian motion can be concatenated via addition $W_{s,u} = W_{s,t} + W_{t,u}$, while concatenating L\'{e}vy areas requires a special case of Chen's relation \citep{Chen1957IntegrationOP}, which uses rescaled L\'{e}vy areas.

\begin{definition}[Rescaled L\'{e}vy areas]
\label{def:theory:rescaled_levy}
Let $0 \leq s < t$ be two times. Then the rescaled space-time L\'{e}vy area $\bar{H}$ and space-time-time L\'{e}vy area $\bar{K}$ are given by
\begin{equation*}
    \bar{H}_{s,t} = (t-s) H_{s,t}, \qquad  \bar{K}_{s,t} = (t-s)^2 K_{s,t}.
\end{equation*}
\end{definition}

\begin{restatable}[Chen's relation]{theorem}{chenrel}
    \label{thm:theory:Chen_rel}
    For times $0 \leq s < t < u$, Chen's relation is given by
    \begin{align}
    \label{eq:theory:Chen_rel}
    \begin{split}
        W_{s,u} &= W_{s,t} + W_{t,u}, \\
        \bar{H}_{s,u} &= \bar{H}_{s,t} + \bar{H}_{t,u} + \frac{u-s}{2} B_t^{s,u}, \\
        \bar{K}_{s,u} &= \bar{K}_{s,t} + \bar{K}_{t,u} + \frac{u-t}{2} \bar{H}_{s,t} - \frac{t-s}{2} \bar{H}_{t,u} + \frac{(u-t)^2 - (t-s)^2}{12} B_t^{s,u}.
    \end{split}
    \end{align}
    where $B_t^{s,u} \coloneqq W_{s,t} - \frac{t-s}{u-s} W_{s,u}$ denotes the Brownian bridge associated with $W$ on $[s,u]$.
\end{restatable}

The proof of this theorem and a visualisation of Chen's relation for space-time L\'{e}vy area are in \cref{sec:proofs:chen}.


\subsubsection{L\'{e}vy area as a singly-indexed proccess}
\label{sec:levy_singly_indexed}

Note that unlike the Brownian motion, which has a value $W_t$ at any time $t$, the L\'{e}vy areas are only defined for time intervals, i.e.~we give a meaning to $H_{s,t}$, but not to $H_t$. Dealing with such a doubly indexed process in numerical applications can be difficult. In particular, when computing $H_{s,t}$, rather than just running \cref{alg:vbt:basic-alg} for $s$ and $t$ separately and then taking the difference, an entirely different approach would have to be devised. Fortunately, we can use Chen's relation to tackle this.

By setting $s=0$ in \cref{eq:theory:Chen_rel}, we obtain
\begin{equation}
\label{eq:theory:chen_from_zero}
\begin{split}
    \bar{H}_{t,u} &= \bar{H}_{0,u} - \bar{H}_{0,t} - \tfrac{1}{2} \! \left( u W_t - t W_u \right), \\[2mm]
    \bar{K}_{t,u} &= \bar{K}_{0,u} - \bar{K}_{0,t} - \frac{u-t}{2} \bar{H}_{0,t} + \frac{t}{2} \bar{H}_{t,u} - \frac{u - 2t}{12} \left( u W_t - t W_u \right).
\end{split}
\end{equation}

Let $Y_{s,r} \coloneqq \big(W_{s,r}, \, H_{s,r}, \, K_{s,r} \big)$ for any $0 \leq s \leq r$. Then, by above, knowing any two of $Y_{0,t}, Y_{0,u}, Y_{t,u}$ lets one compute the third. The same holds when $Y$ is instead just $(W, H)$, or when we use the rescaled version $\bar{Y}_{s,r} = \big( W_{s,r}, \, \bar{H}_{s,r}, \, \bar{K}_{s,r} \big)$. From here on, we will use $Y_{0,r}$ and $Y_r$ interchangeably (likewise for $H$ and $K$).

\subsection{Midpoint Brownian bridge for L\'{e}vy area}
\label{sec:theory:midpoint}

Let $0 \leq s < u$ and $t = \frac{s+u}{2}$. To incorporate L\'{e}vy areas into the \texttt{bridge} function we need to find a way to generate the midpoint values $\bar{Y}_t$ given $\bar{Y}_s$ and $\bar{Y}_u$.
This can be done using theorems 6.1.6 and 6.1.8 from \citep{foster2020a}. This midpoint bridge formula will differ depending on if we are working with just the process $Y = (W, H)$, or with the full $Y = (W, H, K)$. These formulae will be stated separately.

\subsubsection{Space-time L\'{e}vy area only}
\label{sec:theory:midpoint:wh}

Given $W_s, W_u, \bar{H}_s, \bar{H}_u$, we want to generate $W_t$ and $\bar{H}_t$. We first compute
$W_{s,u}, \bar{H}_{s,u}$ using \cref{eq:theory:chen_from_zero}.
The following corollary of \cref{thm:theory:arbitrary_wh} \citep{foster2020a} gives the distribution at the midpoint:
\begin{corollary}[Midpoint Brownian bridge for $\left(W, H\right)$]
\label{cor:theory:midpoint_wh}
    Let $0 \leq s < u$ and $t = \frac{s+u}{2}$. Then
    \begin{align*}
        W_{s,t} &= \frac{1}{2} W_{s,u} + \frac{3}{2} H_{s,u} + Z, &&  \quad
        W_{t,u} = \frac{1}{2} W_{s,u} - \frac{3}{2} H_{s,u} - Z, \\
        H_{s,t} &= \frac{1}{4} H_{s,u} - \frac{1}{2} Z + \frac{1}{2} N, &&  \quad
        H_{t,u} = \frac{1}{4} H_{s,u} - \frac{1}{2} Z - \frac{1}{2} N,
    \end{align*}
    where $Z \sim \normalof{0, \frac{1}{16} (u-s)}$, $N \sim \normalof{0, \frac{1}{12} (u-s)}$, $W_{s,u}$ and $\barH_{s,u}$ are all independent.
\end{corollary}

\begin{proof}
    Follows directly from \cref{thm:theory:arbitrary_wh} by setting $r = t = \frac{s+u}{2}$.
\end{proof}

Since we are working with the rescaled L\'{e}vy area, we can write the above as
\begin{equation}
\label{eq:theory:midpoint_wh_rescaled}
    W_{s,t} = \frac{1}{2} W_{s,u} + \frac{3}{2 (u-s)} \bar{H}_{s,u} + Z, \qquad
    \bar{H}_{s,t} = 
    \frac{1}{8} \bar{H}_{s,u} - \frac{u-s}{4} Z + \frac{u-s}{4} N.
\end{equation}

Finally using Chen's relation again we can compute $W_t, \bar{H}_t$ from $W_s, \barH_s, W_{s,t}$ and $\bar{H}_{s,t}$.

\subsubsection{Both space-time and space-time-time L\'{e}vy areas}
\label{sec:theory:midpoint:whk}

Suppose instead, we are given $(W, \bar{H},  \bar{K})_s$, $(W, \bar{H},  \bar{K})_u$ and we want to generate
$(W, \bar{H},  \bar{K})_t$. We first compute $W_{s,u}, \bar{H}_{s,u}, \bar{K}_{s,u}$ using \cref{eq:theory:chen_from_zero}.
The midpoint Brownian bridge is given by Theorem 6.1.8 from \cite{foster2020a}:

\begin{restatable}[Midpoint Brownian bridge for $(W, H,  K )$]{theorem}{midpointWHK}\label{thm:theory:midpoint_whk}
Let $t$ denote the midpoint $t=s+\frac{1}{2}(u-s)$. Then,
\begin{align*}
W_{s,t} & = \frac{1}{2}W_{s,u} + \frac{3}{2}H_{s,u} + Z, &&  \quad
W_{t,u} = \frac{1}{2}W_{s,u} - \frac{3}{2}H_{s,u} - Z,\\[2pt]
H_{s,t} & = \frac{1}{4}H_{s,u} + \frac{15}{4}K_{s,u} - \frac{1}{2}Z + X_{1}, &&  \quad
H_{t,u} = \frac{1}{4}H_{s,u} - \frac{15}{4}K_{s,u} - \frac{1}{2}Z - X_{1},\\[2pt]
K_{s,t} & = \frac{1}{8}K_{s,u} - \frac{1}{2}X_{1} + X_{2}, && \quad
K_{t,u} = \frac{1}{8}K_{s,u} - \frac{1}{2}X_{1} - X_{2},
\end{align*}
where $W_{s,u}, H_{s,u}, K_{s,u}, Z, X_{1} , X_{2}$ are independent Gaussian random variables with
\begin{equation*}
Z \sim \normalof{0, \frac{u-s}{16}}, \quad X_{1} \sim \normalof{0, \frac{u-s}{768}}, \quad X_{2} \sim \normalof{0, \frac{u-s}{2880}}.
\end{equation*}
\end{restatable}

The proof can be found in \cref{pf:midpoint_whk}.

Expressing this in terms of rescaled L\'{e}vy areas, the formula for $W$ is the same as in \cref{eq:theory:midpoint_wh_rescaled}, and
\begin{equation*}
    \bar{H}_{s,t} = \frac{1}{8} \bar{H}_{s,u} + \frac{15}{8 (u-s)} \bar{K}_{s,u} - \frac{u-s}{4} Z + \frac{u-s}{2} X_1, \quad
    \bar{K}_{s,t} = \frac{1}{32} \bar{K}_{s,u} - \frac{(u-s)^2}{8} X_1 + \frac{(u-s)^2}{4} X_2.
\end{equation*}

Finally, we can use Chen's relation again to compute $W_t, \bar{H}_t, \bar{K}_t$.

\subsection{Interpolation between dyadic points for L\'{e}vy area}
\label{sec:theory:general_times}

To derive the L\'{e}vy-area-endowed version of the $\finalinterp$ function, we need to be able to generate the Brownian motion and L\'{e}vy areas at any intermediate time $r \in [s,u]$.
The $(W, H)$ case is taken from \citep[Theorem 6.1.4]{foster2020a}, whereas the $(W, H, K)$ case is new here.

\subsubsection{Space-time L\'{e}vy area}
\label{sec:theory:general_times:wh}
Given $W_s, W_u, \bar{H}_s, \bar{H}_u$, we want to generate
$W_r $ and $ \bar{H}_r$. We first compute $W_{s,u} $ and $ \bar{H}_{s,u}$ using Chen's relation as in \cref{sec:theory:midpoint:wh}, and rescale to obtain $H_{s,u} = \frac{1}{u-s} \bar{H}_{s,u}$.

Theorem 6.1.4 from \cite{foster2020a} gives the required update rule.

\begin{restatable}[General Brownian bridge for $(W, H)$]{theorem}{arbitraryWH}
Let $r \in [s,u]$ be some time and let $W$ denote a standard Brownian motion. Then\label{thm:theory:arbitrary_wh}
\begin{align*}
    W_{s,r} & = \frac{r-s}{u-s}W_{s,u} + \frac{6(r-s)(u-r)}{(u-s)^{2}}H_{s,u} + \frac{2(a+b)}{u-s}X_{1},\\[3pt]
    H_{s,r} & = \frac{(r-s)^{2}}{(u-s)^2}H_{s,u} - \frac{a}{r-s}X_{1} + \frac{c}{r-s}X_{2},
\end{align*}
where $W_{s,u}, H_{s,u}, X_{1}, X_{2}$ are independent Gaussian random variables with $X_{1}, X_{2}  \sim \mathcal{N}\left(0,1\right)$ and
\begin{align*}
    a  & \coloneqq \frac{\left(r-s\right)^{\frac{7}{2}}\left(u-r\right)^{\frac{1}{2}}}{2(u-s) \, d}, &&
    b \coloneqq \frac{\left(r-s\right)^{\frac{1}{2}}\left(u-r\right)^{\frac{7}{2}}}{2(u-s) \, d},\\[3pt]
    c  & \coloneqq \frac{\sqrt{3}\left(r-s\right)^{\frac{3}{2}}\left(u-r\right)^{\frac{3}{2}}}{6 \, d}, &&
    d \coloneqq \sqrt{\left(r-s\right)^{3}+\left(u-r\right)^{3}}.
\end{align*}
\end{restatable}

Having generated $W_{s,r}$ and $H_{s,r}$ via the theorem above, we use Chen's relation to compute $W_r$ and $\bar{H}_r$.

\subsubsection{Both space-time and space-time-time L\'{e}vy areas}
\label{sec:theory:general_times:whk}
We consider $Y \coloneqq (W, H, K)$, and $\barY = (W, \barH, \barK )$. Then, given $\barY_s$ and $\barY_u$, we want to generate $Y_r$. We first compute $W_{s,u}, \bar{H}_{s,u}, \bar{K}_{s,u}$ using Chen's relation as in \cref{sec:theory:midpoint:whk}, and rescale to obtain $H_{s,u} = \frac{1}{u-s} \bar{H}_{s,u} $ and $ K_{s,u} = \frac{1}{(u-s)^2} \bar{K}_{s,u}$. The following theorem (which is new here) gives a rule for generating $Y_{s,r}$ conditional on $Y_{s,u}$, which in this case is equivalent to conditioning on $Y_s$ and $Y_u$. The proof is in \cref{sec:proofs:arbitrary_whk}


\begin{restatable}[General Brownian bridge for (W, H, K)]{theorem}{arbitraryWHK}
\label{thm:theory:arbitrary_whk}
For any $0 \leq s < u$ and $r \in [s,u]$ the distribution of $Y_{s,r}$ conditioned on $Y_{s,u}$ is Gaussian with the following mean:
\begin{align}
\label{eq:theory:mean_barY_r}
\begin{split}
     \E[W_{s,r} | Y_{s,u}] &= \frac{r-s}{u-s} W_{s,u} + 6 \frac{(r-s)(u-r)}{(u-s)^2} H_{s,u} + 120 \frac{(r-s)(u-r)(\frac{1}{2}(u+s) - r)}{(u-s)^3} K_{s,u}, \\[1mm]
    \E[H_{s,r} | Y_{s,u}] &= \frac{(r-s)^2}{(u-s)^2} H_{s,u} + 30 \frac{(r-s)^2(u-r)}{(u-s)^3} K_{s,u}, \\[1mm]
    \E[K_{s,r} | Y_{s,u}] &= \frac{(r-s)^3}{(u-s)^3} K_{s,u},
\end{split}
\end{align}

and the following covariance matrix:
\begin{align}
\label{eq:theory:general_times:whk:sigma}
\begin{split}
    \vSigma_r \coloneqq \E{\left( Y_{s,r} - \E[Y_{s,r} | Y_{s,u} ] \right) \left( Y_{s,r} - \E[Y_{s,r} | Y_{s,u} ] \right)^T \; \big\vert \; Y_{s,u} } = \begin{bmatrix}
        \Sigma^{\text{WW}}_r & \Sigma^{\text{WH}}_r & \Sigma^{\text{WK}}_r \\
        \Sigma^{\text{WH}}_r & \Sigma^{\text{HH}}_r & \Sigma^{\text{HK}}_r \\
        \Sigma^{\text{WK}}_r & \Sigma^{\text{HK}}_r & \Sigma^{\text{KK}}_r
    \end{bmatrix},
\end{split}
\end{align}
with coefficients
\begin{align*}
    \Sigma^{\text{WW}}_r &= \frac{(r-s) (u-r) \left(\left(2r - u - s\right)^{4} + 4 (r-s)^{2} (u-r)^{2}\right)}{\left(u-s\right)^{5}}, \\[2mm]
    \Sigma^{\text{WH}}_r &= - \frac{(r-s)^{3} (u-r) \left((r-s)^{2} - 3 (r-s) (u-r) + 6 (u-r)^{2}\right)}{2 \left(u-s\right)^{5}}, \\[2mm]
    \Sigma^{\text{WK}}_r &= \frac{(r-s)^{4} (u-r) \left(2r - u - s\right)}{12 \left(u-s\right)^{5}}, \\[2mm]
    \Sigma^{\text{HH}}_r &= \frac{r-s}{12} \left(1 - \frac{(r-s)^{3} \left((r-s)^{2} + 2 (r-s) (u-r) + 16 (u-r)^{2}\right)}{\left(u-s\right)^{5}}\right), \\[2mm]
    \Sigma^{\text{HK}}_r &= - \frac{(r-s)^{5} (u-r)}{24 \left(u-s\right)^{5}}, \qquad
    \Sigma^{\text{KK}}_r = \frac{r-s}{720} \left(1 - \frac{(r-s)^{5}}{\left(u-s\right)^{5}}\right).
\end{align*}

\end{restatable}

\begin{remark}
    $\E[W_{s,r} | Y_{s,u}]$ follows a cubic polynomial, which was first shown in \cite{foster2020optimalPoly} and \cite{habermann2021semicircle}.
\end{remark}

We can generate the multivariate normal $\widehat{Y}_{s,r} \sim \normalof{0, \vSigma_r}$ using a singular-value decomposition, which is not very computationally expensive, since $\vSigma_r$ is only a $3 \times 3$ matrix. A slightly more efficient choice would be the Cholesky decomposition, but $\vSigma_r$ is singular when $r=s$ or $r=u$, so the Cholesky decomposition can become inaccurate when $|r-s|$ or $|u-r|$ are small.

Finally we compute $\left( W_{s,r}, H_{s,r}, K_{s,r} \right) = \widehat{Y}_{s,r} + \E[Y_{s,r} | Y_{s,u}]$, rescale the L\'{e}vy areas and use Chen's relation to compute $W_r, H_r$ and $K_r$.



\section{Implementation}
\label{sec:Implementation}

We can now incorporate the L\'{e}vy-area-endowed versions of the $\bridge$ and $\finalinterp$ functions from Sections \ref{sec:theory:midpoint} and  \ref{sec:theory:general_times} respectively into the \vbt algorithm. The full pseudocode for $\bridge$ and $\finalinterp$ can be found in \cref{appendix:updated_bridge_finalinterp}.

In the interest of efficiency and numerical stability we also introduce two further changes to the algorithm:
\begin{itemize}
    \item Instead of $Y_s$ and $Y_u$, we keep track of $\bar{Y}_s, \bar{Y}_u$ and $\bar{Y}_{s,u}$\footnote{Note that this is just for computational convenience inside a single evaluation of the VBT, and should not be confused with computing $Y_{r_0, r_1}$. In that case we first separately evaluate the VBT for $Y_{r_0}$ and $Y_{r_1}$, before computing $Y_{r_0, r_1}$ from these.} which lets us avoid repeatedly rescaling the L\'{e}vy areas and saves us an additional application of Chen's relation in the $\bridge$ function.
    \item We internally normalise the interval of definition (see \cref{sec:IntervalNormalisation}).
\end{itemize}

\subsection{Interval Normalisation}
\label{sec:IntervalNormalisation}

In order to make the calculations simpler, as well as more efficient and numerically stable, we modify the \vbt so that its interval of definition $[t_0, t_1]$ is internaly normalised to $[0,1]$. The tolerance for discretising the interval is also rescaled accordingly. When the Brownian motion and Lévy areas are evaluated at a given time $r$, this is first rescaled to $\hat{r} \coloneqq \frac{r - t_0}{t_1 - t_0}$. After the computations are done in the normalized space, the results are then rescaled back from $[0,1]$ to the original interval $[t_0, t_1]$ using Brownian scaling, i.e.~$X_r^{\text{out}} = \sqrt{t_1 - t_0} X_{\hat{r}}$ for $X \in \{W, H, K\}$.

This provides three main benefits:
\begin{itemize}
    \item With normalisation, all the times appearing in the tree are exact dyadic rationals, and are thus represented exactly in floating-point format. This eliminates many potential sources of error.
    \item The midpoint update rule has a nicer form and is about $20 \%$ faster to compute.
    \item If $t_1$ and $t_0$ are very large, but close together, the original implementation could suffer from catastrophic cancellation errors. Without normalisation, these would appear in many calculations as we descend the tree, causing significant compounding errors. These are avoided with our construction.
    
\end{itemize}

\subsection{The main loop}

Our algorithm depends on a \texttt{levy\_area} parameter, governing which types of L\'{e}vy area will be computed. We present the case when \texttt{levy\_area = `space-time-time'}, i.e.~when all of $W, H $ and $ K$ are being computed. The $(W, H)$ case with \texttt{levy\_area = `space-time'} is obtained by simply removing all references to $K$, or removing both $H$ and $K$ when \texttt{levy\_area = `none'}, so only Brownian increments are generated.

We introduce a helper function \texttt{rescale} which converts $\bar{Y} = (W, \bar{H}, \bar{K})$ to $Y = (W, H, K)$.

\begin{algorithm}[H]
\caption{\texttt{rescale}}\label{alg:rescale}

\hspace*{2.5mm} \textbf{Input:} $h$ - time increment, $\bar{Y}$ - Brownian motion and L\'{e}vy areas
\vspace{1mm}
\begin{algorithmic}
\State $(W, \bar{H}, \bar{K}) \gets \bar{Y}$
\State $H \gets h^{-1} \bar{H}, \;\; K \gets h^{-2} \bar{K}$
\State \Return $(W, H, K)$
\end{algorithmic}
\end{algorithm}

Now we can state the algorithm for evaluating the \vbt  at a single time $r$.

\begin{algorithm}[H]
\caption{\vbteval with space-time-time L\'{e}vy area}\label{alg:main}

\hspace*{2.5mm} \textbf{Input:} $r$ - evaluation time, $[t_0, t_1]$ - interval of definition, $\rho$ - PRNG seed, $d$ - dimension, $\varepsilon$ - tolerance.
\vspace{1mm}

\hspace*{2.5mm} \textbf{Result:} Approximation to $\big(W_{t_0,r}, \, H_{t_0,r}, \, K_{t_0,r} \big)$
\vspace{1mm}

\begin{algorithmic}
\setstretch{1.5}
\State $\hat{r} \gets \frac{r - t_0}{t_1 - t_0}$
\State $l \gets 0, \;\; s \gets 0, \;\; u \gets 1$ \Comment{$l$ is the current level in the tree and $u-s = 2^{-l}.$}
\State $\rho, \widehat{\rho}_W, \widehat{\rho}_H, \widehat{\rho}_K \gets \texttt{split\_seed}(\rho, 4)$
\State $W_u \gets \normalof{0, \eye_d, \widehat{\rho}_W}, \; H_u \gets \normalof{0, \frac{1}{12} \eye_d, \widehat{\rho}_H}, \; K_u \gets \normalof{0, \frac{1}{720} \eye_d, \widehat{\rho}_K} $
\State $\bar{Y}_s \gets (0, 0, 0), \;\; \bar{Y}_u \gets (W_u, H_u, K_u), \;\; \bar{Y}_{s,u} \gets \bar{Y}_u$
\While{$2^l > \frac{\varepsilon}{t_1 - t_0}$}
    \State $\rho_1, \rho_2, \widehat{\rho} \gets \texttt{split\_seed}(\rho, 3)$
    \State $t \gets s + 2^{l+1}$
    \State $\bar{Y}_t, \, \bar{Y}_{s,t}, \, \bar{Y}_{t,u},  \gets \bridge  \big(l, s, \bar{Y}_s, \bar{Y}_u, \bar{Y}_{s,u}, \widehat{\rho}\, \big)$
    
    {\setstretch{1.1}
    \If{ $\hat{r} \leq t$ }
        \State $\rho \gets \rho_1$    \Comment{Descend into the \textbf{left} child.}
        \State $s \gets s, \; u \gets t$
        \State $\bar{Y}_s \gets \bar{Y}_s, \; \bar{Y}_u \gets \bar{Y}_t, \; \bar{Y}_{s,u} \gets \bar{Y}_{s,t}$
    \Else
        \State $\rho \gets \rho_2$  \Comment{Descend into the \textbf{right} child.}
        \State $s \gets t, \; u \gets u$
        \State $\bar{Y}_s \gets \bar{Y}_t, \; \bar{Y}_u \gets \bar{Y}_u, \; \bar{Y}_{s,u} \gets \bar{Y}_{t,u}$
    \EndIf
    }
    \State $l \gets l + 1$
\EndWhile

\State $\bar{Y}_{\hat{r}} \gets \finalinterp \big(s, r, u, \bar{Y}_s, \bar{Y}_u, \bar{Y}_{s,u}, \rho \big)$
\State $Y_{\hat{r}} \gets \texttt{rescale}(\hat{r}, \, \bar{Y})$ 
\State \Return $\sqrt{t_1 - t_0} \, Y_{\hat{r}}$ \Comment{Undo the effects of interval normalisation.}
\end{algorithmic}
\end{algorithm}

The usual use case is when instead of a single evaluation time $r$, we are provided with an interval $[r_0, r_1] \subseteq [t_0, t_1]$ and want the value of $Y_{r_0, r_1}.$ In this case, we first compute $\bar{Y}_{\hat{r}_0}$ and $\bar{Y}_{\hat{r}_1}$ using \cref{alg:main} (stopping at the line with \finalinterp) and apply Chen's relation (\cref{eq:theory:Chen_rel}) to obtain $\bar{Y}_{\hat{r}_0, \hat{r}_1}$. We then compute $Y_{\hat{r}_0, \hat{r}_1} = \texttt{rescale} \big( \hat{r}_1 - \hat{r}_0, \, \bar{Y}_{\hat{r}_0, \hat{r}_1} \big)$
and finally undo the normalisation: $Y_{r_0, r_1} = \sqrt{t_1 - t_0} Y_{\hat{r}_0, \hat{r}_1}.$

\subsection{Matching the true distribution}
\label{sec:impl:true_distn}

Due to our careful construction of the $\finalinterp$ function, we can now prove the general case of \cref{prop:vbt:match_dist}, which also holds when L\'{e}vy area is present. We have already proven in \cref{sec:theory:general_times} that conditional on $Y_s$ and $Y_u$, the output of $\finalinterp  \big(l, s, \bar{Y}_s, \bar{Y}_u, \bar{Y}_{s,u}, \rho \big)$ has the correct distribution. Theorems \ref{cor:theory:midpoint_wh} and \ref{thm:theory:midpoint_whk} further show that for all points in $V = \set{t_0 + (t_1 - t_0) k 2^{-L} : k \in \set{0, 1, \ldots , 2^L} }$ the generated random variables $\left( Y_v \right)_{v\in V}$ have the correct joint distribution.

We first need the following lemma:
\begin{lemma}[$Y$ is a Markov process]
\label{lem:impl:markovian}
    The processes $( W_t, H_t, K_t )_{t \geq 0}$, $( W_t, \bar{H}_t, \bar{K}_t )_{t \geq 0}$, $( W_t, H_t )_{t \geq 0}$, and $( W_t, \bar{H}_t )_{t \geq 0}$ are all Markov processes.
\end{lemma}
\begin{proof}
    Rearranging the definitions of $\barH$ and $\barK$ we obtain
    \begin{align*}
        \barH_t &= \int_0^t W_s \, ds \, - \, \frac{t}{2} W_t, \\[2mm]
        \barK_t &= \frac{t}{2} \int_0^t W_s \, ds - \int_0^t s W_s \, ds - W_t \int_0^t \frac{s}{t} \left( \frac{t}{2} -s \right) ds.
    \end{align*}
    Using this, we can write the process $\left( W_t, H_t, K_t \right)$ as an SDE:
    \begin{align}
    \label{eq:impl:whk_sde}
    \begin{split}
        dW_t &= dW_t, \quad d\bar{H}_t = \frac{W_t}{2} \, dt - \frac{r}{2} \, dW_t, \\
        d\bar{K}_t &= \Big( \barH_t - \frac{7}{12} t W_t \Big) dt + \frac{t^2}{12} \, dW_t.
    \end{split}
    \end{align}
    Since all of the SDE coefficients only depend on the current state of $\left( W_t, H_t, K_t \right)$, the process is Markovian.

    Moreover, as $\big(W_r, H_r, K_r \big)$ is a deterministic function of $\big( W_r, \bar{H}_r, \bar{K}_r \big)$, it is also Markovian. The same arguments apply for the pair $(W, H)$.
\end{proof}

We are now ready to prove the general version of \cref{prop:vbt:match_dist}.

\begin{theorem}[Distributional correctness of the Virtual Brownian Tree with L\'{e}vy  area]

\label{prop:impl:match_dist}

    Let $Y$ be any of $W$, $(W, H)$, or $(W, H, K)$ and let \texttt{levy\_area} be set to \texttt{`none'}, \texttt{`space-time'}, or \texttt{`space-time-time'} respectively. Let $0 \leq t_0 < t_1$, $L, d \in \naturals$, $\varepsilon = (t_1 - t_0) 2^{-L}$ and let $\rho$ be a PRNG seed. Define the generated path $y$ as
    \[
    y_r \coloneqq \vbteval \! \left(r, [t_0, t_1], \rho, d, \varepsilon, \texttt{levy\_area} \right).
    \]

    Let $V = \set{t_0 + (t_1 - t_0) k 2^{-L} : k \in \set{0, 1, \ldots , 2^L} }$ be the set of vertices of the VBT and let $S = \set{r_1, \ldots , r_N} \subset [t_0, t_1]$ be a set of query-times such that for all $0 \leq n < N$ we have $r_n \leq r_{n+1}$ and $ \left\vert V \cap [r_n, r_{n+1}] \right\vert \geq 1$.

    Suppose that the seed $\rho$ is chosen uniformly at random (but is fixed throughout). Then the joint distribution of $y$ sampled at points from $S$ matches the joint distribution of $Y$ sampled at the same points. That is,
    \begin{equation*}
        \left( y_r \right)_{r \in S} \eqd \left( Y_r \right)_{r \in S}.
    \end{equation*}
\end{theorem}

\begin{proof}
    By the monotone class theorem, it is enough to prove that for any choice of measurable sets \\ $G_1, \ldots, G_N \subseteq \reals^d$ we have $\P \! \left( y_{r_n} \in G_n \text{ for } n = 1, \ldots N \right) = \P \! \left( Y_{r_n} \in G_n \text{ for } n = 1, \ldots N \right)$
    or equivalently
    \[
    \E{\prod_{n=1}^N \indicator_{G_n} \left(y_{r_n} \right)} = \E{\prod_{n=1}^N \indicator_{G_n}\left(Y_{r_n} \right)}.
    \]
    
    For $r \in [t_0, t_1]$, we write $\topV{r} \coloneqq \min \set{v \in V: v \geq r}$ and $\botV{r} \coloneqq \max \set{v \in V: v \leq r}$ (note that these are always well defined since $t_0, t_1 \in V$).

    We will make use of the following facts:
    \begin{enumerate}
        \item Due to our PRNG-splitting method, the keys at the leaves of the \vbt are independent and hence if $t_0 \leq q < r \leq t_1$ and $V \cap [q, r]$ is non-empty, then $y_r$ and $y_q$ are conditionally independent given $y_{\botV{q}}, y_{\topV{q}}, y_{\botV{r}}, y_{\topV{r}}$.
        \item If $r \in [t_0, t_1]$ and $s = \botV{r}, u = \topV{r}$ then by our choice of interpolation (see theorems \ref{thm:theory:arbitrary_wh} and \ref{thm:theory:arbitrary_whk})
        \[
            \left( y_r | y_s, y_u \right) \eqd \left( Y_r | Y_s = y_s, Y_u = y_u \right).
        \]
        \item By theorems \ref{cor:theory:midpoint_wh} and \ref{thm:theory:midpoint_whk}, we have $\left( y_v \right)_{v \in V} \eqd \left( Y_v \right)_{v \in V}.$
        \item By \cref{lem:impl:markovian}, $Y$ is Markovian, so for any $t \geq 0$, the processes $(Y_s)_{s\leq t}$ and $(Y_{t,s})_{s\geq t}$ are independent conditional on $Y_t$. This implies that $Y$ has the same property as was shown for $y$ in fact 1 above.
    \end{enumerate}

    Using the above properties, we have
    \begin{align*}
        \E{\prod_{n=1}^N \indicator_{G_n} \left(y_{r_n} \right)} &= \E{ \E{\prod_{n=1}^N \indicator_{G_n} \left(y_{r_n} \right) \;\Big\vert\; \left( y_v \right)_{v \in V} } } && \text{(tower law)}\\
        &= \E{ \prod_{n=1}^N \E{\indicator_{G_n} \left(y_{r_n} \right) \;\Big\vert\; y_{\botV{r_n}}, y_{\topV{r_n}} } } &&  \text{(fact 1)}\\
        &= \E{ \prod_{n=1}^N \E{\indicator_{G_n} \left(Y_{r_n} \right) \;\Big\vert\; Y_{\botV{r_n}} = y_{\botV{r_n}}, Y_{\topV{r_n}} = y_{\topV{r_n}} } } &&  \text{(fact 2)}\\
        &= \E{ \prod_{n=1}^N \E{\indicator_{G_n} \left(Y_{r_n} \right) \;\Big\vert\; Y_{\botV{r_n}}, Y_{\topV{r_n}} } } &&  \text{(fact 3)}\\
        &= \E{ \E{\prod_{n=1}^N \indicator_{G_n} \left(Y_{r_n} \right) \;\Big\vert\; \left( Y_v \right)_{v \in V} } } = \E{ \prod_{n=1}^N \indicator_{G_n} \left(Y_{r_n} \right)} && \text{(fact 4 \& tower law)}
    \end{align*}
    
\end{proof}

\begin{remark}
    The key assumption here is that there is at most one sample point between each two consecutive dyadic points $v, v' \in V$, i.e.~$\left| S \cap [v, v'] \right| \leq 1.$ Without this requirement, we can find a simple counterexample. Let $q, r \in (v, v')$ and for simplicity consider the case $Y=W$. When evaluating $w_r$ and $w_q$, the same key is used in the call to $\finalinterp$, and hence the $Z$ in the computation of $w_r$ is the same as that in $w_q$. Therefore, knowing $w_v, w_{v'}$ and $w_r$ lets us deterministically compute the value of $w_q$. However, $W_q$ is not deterministic after we fix $W_v, W_{v'}$ and $W_r$. Therefore, the joint distribution of $(w_r, w_q)$ cannot be the same as that of $(W_r, W_q)$.
\end{remark}


\section{Application to adaptive SDE solvers}
\label{sec:application}
To conclude, we will demonstrate how the combination of adaptive time-stepping and high order solvers (which rely on Lévy areas) can result in significant speed-ups for SDE simulation.

\paragraph{VBT time complexity.} The Virtual Brownian Tree is based on binary search, so a single evaluation takes $\bigO{ \log \! \left( \frac{1}{\varepsilon} \right)}$ time, where $\varepsilon$ is the user-specified tolerance. If we are using an adaptive solver with a minimal step size of $h_{\text{min}}$, \cref{prop:impl:match_dist} tells us that as long as $\varepsilon \leq h_{\text{min}}$, then all Brownian samples drawn by the solver will have exactly the correct joint distribution. Therefore, if the solver takes $n$ computational steps, the total time complexity of a solve will be $\bigO{- \log(h_{\text{min}}) \, n}$.

\subsection{Cox-Ingersoll-Ross model}
\label{sec:application:cir}

Consider the Cox-Ingersoll-Ross (CIR) model from mathematical finance, in Stratonovich form:
\begin{equation}
    d X_t = a (\tilde{b} - X_t) dt + \sigma \sqrt{X_t} \circ dW_t, \quad \tilde{b} \coloneqq b - \frac{\sigma^2}{4 a}
\end{equation}
where $a, b, \sigma > 0$ are the mean reversion speed, mean reversion level, and volatility parameters respectively, and $W$ is a one-dimensional Brownian motion. It was proven by \cite{hefter2019slowCIR} that for any constant step size solver, its SOC (recall \cref{def:intro:strong_order}) can be made arbitrarily small by increasing the volatility $\sigma$.

The convergence of adaptive solvers is still an open question. In \cite{gaines97variable} and \cite{foster2023convergence} the authors prove that adaptive solvers converge to the correct limit given some assumptions about the SDEs or the step-size controller, but \cite{foster2023convergence} also present a counterexample when these assumptions are violated. Unfortunately, these theorems do not apply to the CIR model due to a non-Lipschitz diffusion coefficient -- so we will empirically check the correctness of convergence.

In our experiment, we compare two solvers: the Drift-Implicit Euler (DIE) solver by \cite{Alfonsi2005OnTD} and a High-order Strang splitting (HoSts) solver by \cite{foster2024highSIAM}, both with constant and adaptive time-stepping. In addition to being high-order, the HoSts solver can ensure that its solution is always non-negative as long as $\tilde{b} \geq 0$.

To design our step-size controller, we use the mean square error propagation principle \citep{foster2023convergence}, which asserts that unlike ODEs, where all $L_p$ errors propagate linearly, in SDEs it is the mean square errors that propagate linearly. Hence, to control the global error, we bound the local errors by $\epsilon_{\text{local}}^2 \leq C \, h \, \varepsilon^2$
where $h$ is the length of the current step, $\varepsilon$ is the global error tolerance and $C$ is a proportionality constant.
While $\epsilon_{\text{local}}$ can be estimated using embedded methods \citep{foster2023convergence}, or half-stepping,\footnote{Half stepping means that in each step, the solver makes a coarse computation over the whole step, as well as a finer computation with two half-steps, and takes the difference of the results as the local error.} we use an analytic estimate specific to CIR \citep{foster2020a}:  $\epsilon_{\text{local}} \propto \big| \frac{a \tilde{b} \sigma^2}{2 X_t} \big| \, h^2$.
This gives rise to a simple step-size controller, which sets $h \approx \left( X_t \, \varepsilon \right)^{\frac{2}{3}}.$ The code for our numerical experiments is available at \textcolor{blue}{\href{https://github.com/andyElking/Single-seed_BrownianMotion}{github.com/andyElking/Single-seed\_BrownianMotion}}.
\begin{figure}[H]
\begin{center}
\begin{subfigure}{.5\textwidth}
  \centering
  \includegraphics[width=.95\linewidth]{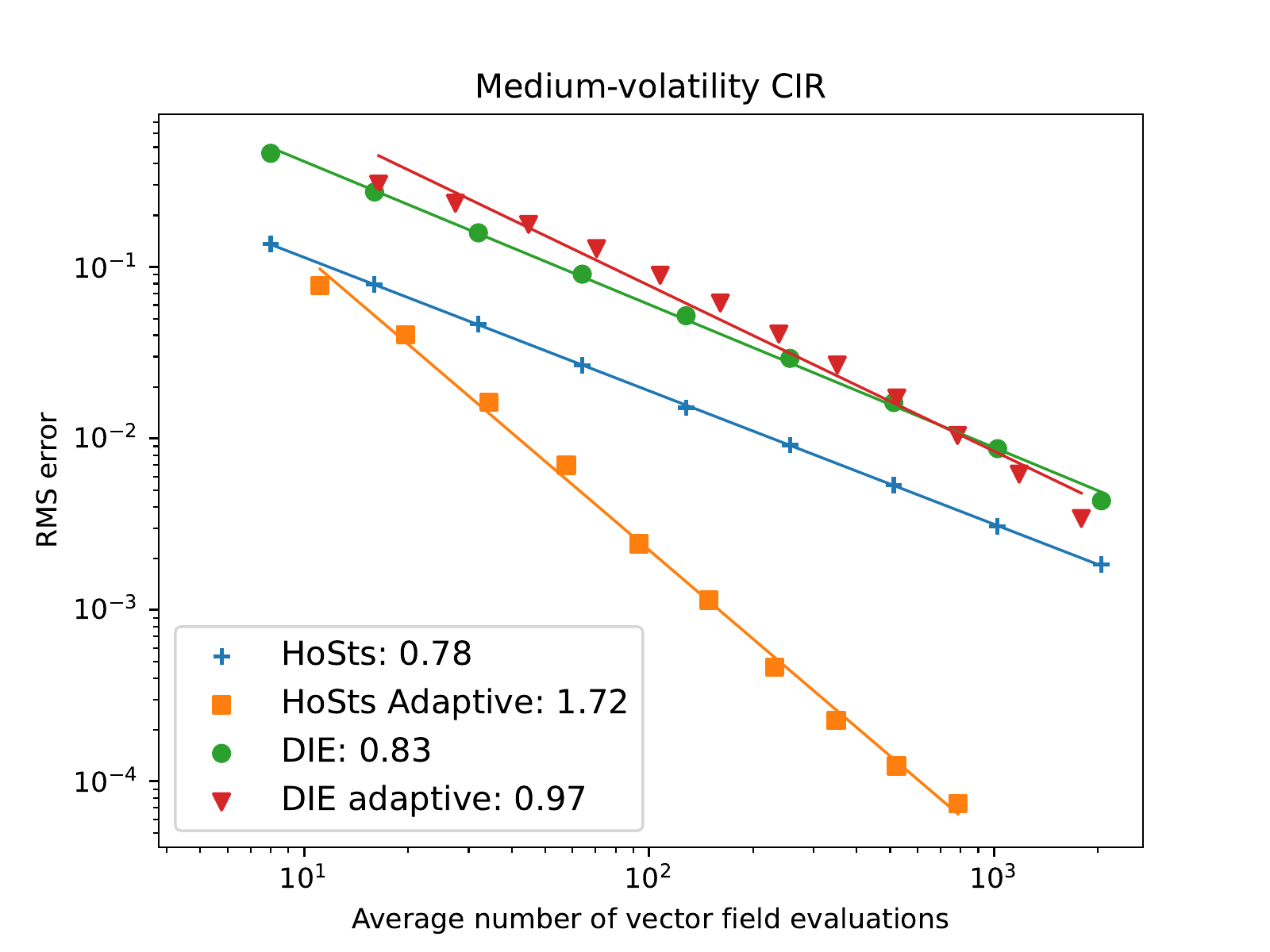}
\end{subfigure}%
\begin{subfigure}{.5\textwidth}
  \centering
  \includegraphics[width=.95\linewidth]{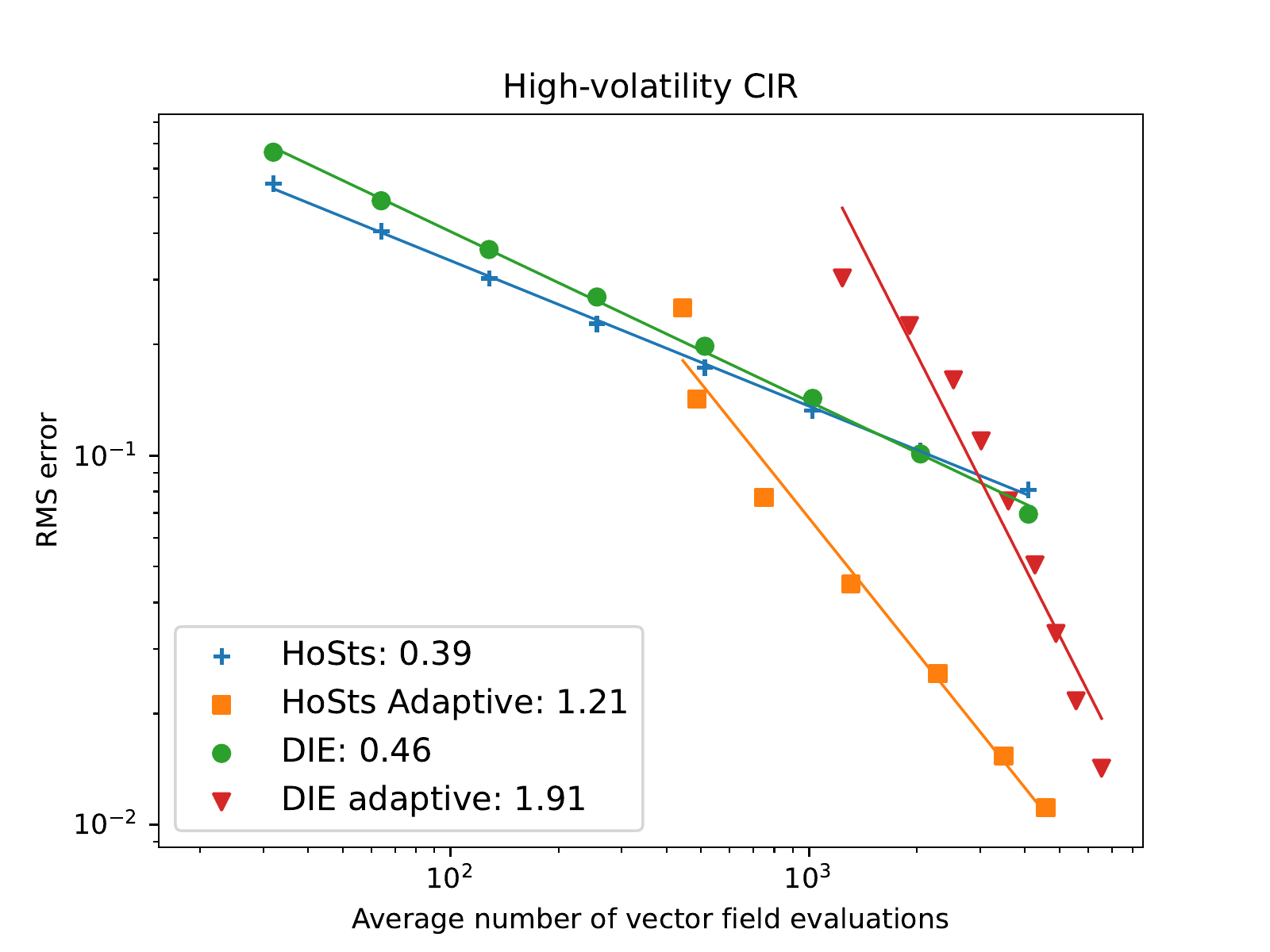}
\end{subfigure}
\caption{Comparison of the HoSts solver and Drift-Implicit Euler applied to the CIR process both with constant step size and adaptive time-stepping. We use $a=b=1$ with $\sigma = 1.5$ (left) and $\sigma = 2.5$ (right). The SOC of each solver is written in the legend. The results are averaged over 1000 random seeds.}
\label{fig:cir_order}
\end{center}
\end{figure}

\subsection{Langevin Monte Carlo}
\label{sec:application:langevin}

To demonstrate the combined capabilities of space-time-time Lévy area and adaptive time-stepping, we turn to a Markov Chain Monte Carlo (MCMC) algorithm based on the underdamped (or kinetic) Langevin diffusion (ULD). Using ULD for MCMC has recently received significant attention \citep{cheng18ULDMC, rioudurand2023MALT, DalRiou2020OnSampling,sanz-serna21wassersteinULD}, as it is the stochastic equivalent of Hamiltonian Monte Carlo (HMC) \citep{chen2014stochastic}. ULD is given by:
\begin{equation}
\label{eq:uld}
\begin{split}
    d \mathbf{x}_t &= \mathbf{v}_t \, dt \\
    d \mathbf{v}_t &= - \gamma \, \mathbf{v}_t \, dt - u \,
    \nabla \! f( \mathbf{x}_t ) \, dt + \sqrt{2 \gamma u} \, d W_t,
\end{split}
\end{equation}
where $\vx_t, \vv_t \in \R^d$, $W$ is a $d$-dimensional BM, $\gamma, u \in \R$ are the friction and damping coefficients and $f: \R^d \rightarrow \R$ is a potential function.
The stationary density of this SDE is $p(x) = C \exp(-f(x))$ where $C$ is a normalising constant.

ULD is a rare type of SDE for which a third order solver has been found -- namely the QUadrature Inspired and Contractive Shifted ODE with Runge-Kutta Three (QUICSORT) solver by \cite{scott2025quicsort}, which requires space-time-time Lévy area as input and can be used adaptively via half-stepping. We compare QUICSORT to the Euler-Maruyama method, which is the usual discretisation used for Langevin Monte Carlo, as well as to SRA1 -- a method developed for additive-noise SDEs by \cite{rossler2010runge}. We also compare to the No U-Turn Sampler by \cite{hoffman2011NUTS} which is often considered to be the gold standard for MCMC. For the adaptive stepping, we use a PI controller with $K_{P} = 0.1$ and $K_{I} = 0.4$ which are similar to the recommended values for SDEs by \cite{ilie2015adaptive}. We also use Diffrax's default value of 0.01 for the initial step size considered by the PI controller. To obtain the graph on the right of \cref{fig:qualitative_funnel}, we set the relative tolerance to $0$ and vary the absolute tolerance.

\subsubsection{Neal's Funnel}

As the first target distribution we chose a common MCMC benchmark called Neal's funnel \citep{neal2003}, given by the pair $x \sim \mathcal{N}(0, 9), \; y \sim \mathcal{N}(0, \exp(x) \mathbf{I}_{9})$. This is particularly challenging for Markov chain algorithms due to the very narrow funnel, which is hard for some samplers to enter into.

QUICSORT with constant steps seems to perform best in the strong convergence graph in \cref{fig:qualitative_funnel}, partly because the adaptive version uses half-stepping, which requires 3-times more vector field evaluations per step. However, constant stepping suffers from instabilities which produce points far outside the funnel. Despite removing all outliers with $\Vert z \Vert > 100$, constant-step solvers underperformed in \cref{tab:neals_funnel_results}. Furthermore, with adaptive stepping the solver automatically tries to use the minimal number of steps needed to achieve the user-specified precision. Hence, adaptive QUICSORT uses the fewest $\nabla f$ evaluations in \cref{tab:neals_funnel_results}.


\begin{figure}[H]
    \centering
    \begin{subfigure}{0.55\textwidth}
        \begin{subfigure}{0.49\linewidth}
            \centering
            \caption{\small{Euler constant step size}}
            \includegraphics[width=0.95\linewidth]{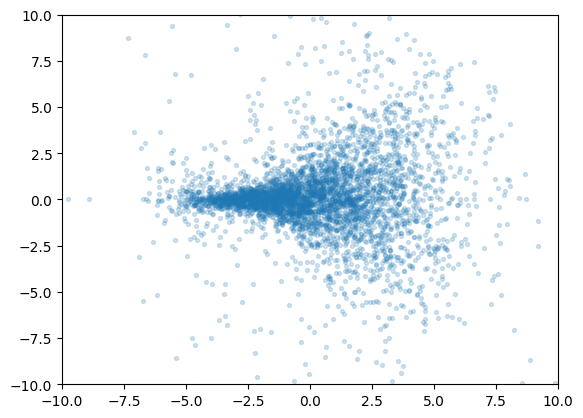}
        \end{subfigure}
        \begin{subfigure}{0.49\linewidth}
            \centering
            \caption{\small{NUTS}}
            \includegraphics[width=0.95\linewidth]{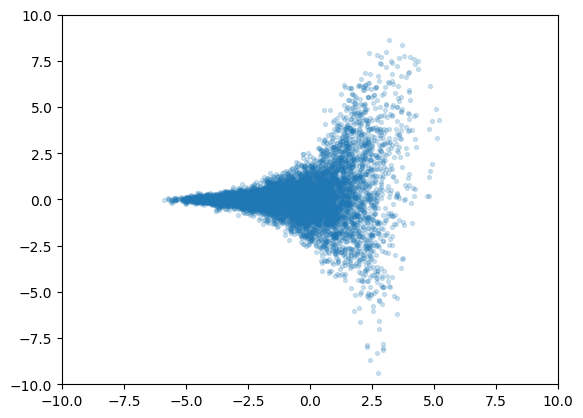}
        \end{subfigure}
        \begin{subfigure}{0.49\linewidth}
            \centering
            \caption{\small{QUICSORT adaptive}}
            \includegraphics[width=0.95\linewidth]{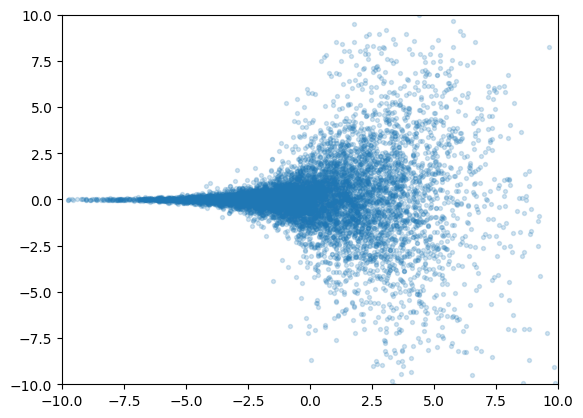}
        \end{subfigure}
        \begin{subfigure}{0.49\linewidth}
            \centering
            \caption{\small{True distribution}}
            \includegraphics[width=0.95\linewidth]{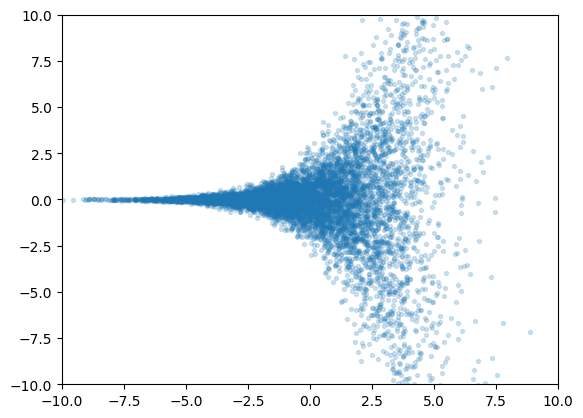}
        \end{subfigure}
    \end{subfigure}%
    \begin{subfigure}{0.45\textwidth}
        \centering
        \includegraphics[width=0.97\linewidth]{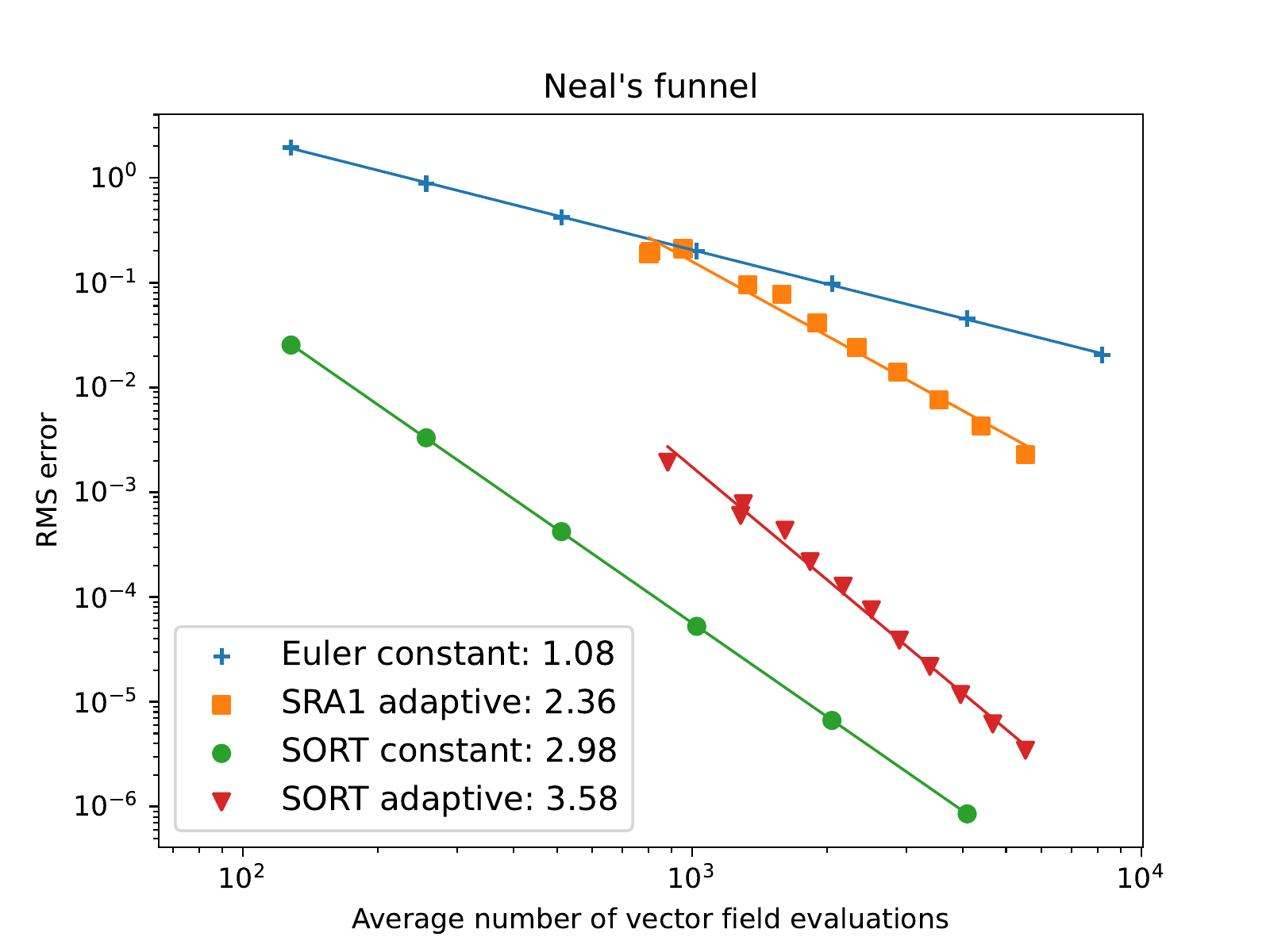}
    \end{subfigure}
    \caption{Left: qualitative comparison of Neal's funnel distributions generated by different methods. \\ Right: The SOC of different solvers on the Neal's funnel Langevin SDE.}
    \label{fig:qualitative_funnel}
\end{figure}

\begin{table}[H]
\centering
\caption{A quantitative comparison of NUTS, and ULD MCMC (with constant or adaptive time-stepping) for sampling from Neal's funnel. The samples were generated using 64 chains of 128 samples, with $\Delta t = 2$ between samples and 16 burn-in iterations. We rescaled the samples by $x' = x/3, \; y' = y e^{-x/2}$, which should result in a 10-dim standard Gaussian. We computed the absolute errors in the mean and the covariance matrix -- we report the maximum and average of the entries. For each marginal we performed the Kolmogorov-Smirnov (KS) test and computed the effective sample size (ESS). Results were averaged over 5 random seeds.}
\label{tab:neals_funnel_results}
\begin{tabular}{cccccccccc}
\toprule
\multirow{ 2}{*}{Method} & \multirow{ 2}{*}{\begin{tabular}{c}$\nabla f$ evals \\ per sample \end{tabular}} & \multicolumn{2}{c}{Mean error} & \multicolumn{2}{c}{Cov error} & \multirow{ 2}{*}{\begin{tabular}{c}Avg KS \\ P-values \end{tabular}} & \multicolumn{2}{c}{ESS}\\
 & & Max & Avg & Max & Avg & & Min & Avg\\
\midrule
NUTS                & 733.7             & 0.71              & 0.40          & 0.49  & 0.14  &  \textbf{0.14} & 0.007 & 0.015 \\
ULD Euler const     & 137.0             & 1.85              & 0.59          & 8e5   & 4e4   & 0.0   & 0.14  & \textbf{0.33}  \\
ULD QUICSORT const  & 274.0             & 0.12              & 0.042         & 1e6    & 6e4   & 0.11   & 0.10  & 0.21  \\
ULD QUICSORT adap   &  \textbf{88.6}    & \textbf{0.073}    & \textbf{0.037}& \textbf{0.40} & \textbf{0.026} & 0.06  & \textbf{0.23}  & 0.24  \\
\bottomrule
\end{tabular}
\end{table}

\subsubsection{Logistic Regression}

We also tested this third-order adaptive Langevin Monte Carlo method on a Bayesian logistic regression task using real-world data. In this task we are given data $x_i \in \reals^d$ and labels $y_i \in \set{0,1}$ (for $i = 1, \ldots, n$) and the goal is to find weights $\theta \in \reals^d,$  such that the relation $\P\big(y_i = 1 | x_i, \theta \big) = \left( 1 + \exp(-\theta^T x_i) \right)^{-1}$ closely approximates the true distribution of the data (in KL divergence). We also include a bias term, so $d$ is 1 larger than the dimensionality of the dataset itself. While the optimal weights can be obtained with simple gradient descent, our aim is to sample from the Bayesian posterior distribution $p(\theta | \DD) = \frac{p(\DD|\theta) p(\theta)}{p(\DD)}$, where $p(\theta) = \normalof{\vzero, \var_{i=1,\ldots,n}(x_i)}$ is the prior and $\DD$ is the data. Computing the posterior is usually infeasible because $p(\DD)$ involves an intractable integral, but we can nevertheless sample from the posterior using MCMC or LMC as above.

We use data from \cite{li2023mied}, which includes 13 datasets of dimensionalities between $d=3$ and $d=61$. We also use the ``Isolet" dataset ($d=617$) from \cite{isolet}, which has 26 classes, but we only use two of those for our logistic regression task. While NUTS (which is considered the gold standard for a good reason) converges faster on most of the low-dimensional datasets, QUICSORT significantly outperforms NUTS on the high dimensional ones (``splice" with $d=61$ and ``Isolet" with $d=617$) and does slightly better on some low dimensional ones as well (``flare solar" with $d=10$, ``image" with $d=19$ and ``titanic" with $d=4$). Furthermore, the adaptive time-stepping version of QUICSORT was able to find the near-optimal step lengths needed for the desired accuracy and thus reached the same accuracy as constant-step QUICSORT, but with fewer evaluations of $\nabla \log p(\theta | \DD)$.

\begin{figure}
    \centering
    \begin{subfigure}{.33\textwidth}
      \centering
      \caption{flare solar ($d=10$)}
      \includegraphics[width=\linewidth]{Images/flare_solar_final.pdf}
    \end{subfigure}%
    \begin{subfigure}{.33\textwidth}
      \centering
      \caption{splice ($d=61$)}
      \includegraphics[width=\linewidth]{Images/splice_final.pdf}
    \end{subfigure}%
    \begin{subfigure}{.33\textwidth}
      \centering
      \caption{isolet ($d=617$)}
      \includegraphics[width=\linewidth]{Images/isolet_ab_final.pdf}
    \end{subfigure}
    \caption{For each of the 14 datasets we ran $2^{15}$ independent chains of NUTS and the same number of Langevin trajectories simulated using the QUICSORT solver. At regular intervals along the run we compared these $2^{15}$ particles against the ``ground truth" samples, which were obtained using NUTS with a sufficiently long burn-in period. We plot the energy distance \citep{szekely2002energy} between the obtained samples and the ground truth samples on the y-axis and the number of times $\nabla f$ was evaluated up to that point in the run on the x-axis. We repeated the experiment five times and presented the standard deviation as the shaded region in the plot. We present 3 of the 14 plots here and the rest can be found at the GitHub link above.}
    \label{fig:logreg convergence plot}
\end{figure}

\section{Discussion and open directions}

The Langevin Monte Carlo algorithm based on the third order solver QUICSORT in \cref{sec:application:langevin} yielded surprisingly good results using very few gradient evaluations. While the QUICSORT algorithm is probably not Metropolis adjustable, there is potential to combine other MCMC approaches with adaptive stepping. For example in \cite{Leroy2024AdaptiveLangevin}, the authors rescale the time-dimension of the Langevin equation in a way that makes time pass slower in regions of higher volatility. However, unlike ours, their approach does not monitor how much error the solver makes, so combining their Monte Carlo techniques with our improved adaptive time-stepping approach could lead to promising results. 

Another possible future application of our method is Multilevel Monte Carlo \citep{Giles08MLMC}, which has long been a staple in the MCMC toolbox. It relies on being able to solve an SDE twice with the same underlying Brownian path, but with different discretizations. This can easily be accomplished when time-steps are constant and one discretization is exactly twice as fine as the other, but if we are using the standard way of generating Brownian motion it would be infeasible to combine this with adaptive time-stepping. However, since VBT always generates the same path given a fixed seed, our method enables combining MLMC with both adaptive time-stepping and high-order SDE solvers.

MCMC methods normally rely on having very long chains, since they exactly match the target (stationary) distribution only at $t=\infty$. Getting a good estimate on how close the chain is to stationarity is important for determining the burn-in period. Chain couplings were used by \cite{biswas2019estimating} for estimating convergence and by \cite{jacob2019unbiased} to achive stationarity at a finite time. \cite{chada2023unbiased} extends this to underdamped/kinetic Langevin Monte Carlo by using coupled SDE trajectories which start at different points and use different (constant) step sizes, but are driven by the same noise. Adding adaptive stepping to this construction is a promising application for the Virtual Brownian Tree, since it can guarantee the same Brownian path is provided to both chains, even if the solver uses different step sizes on each of them.

\subsection*{Acknowledgements}
AJ thanks the UK National Cyber Security Centre for funding his research and Extropic.ai for funding his work on implementing the Virtual Brownian Tree and high order SDE solvers in Diffrax. AJ and JF would also like to acknowledge support from the University of Bath, the Maths4DL programme under EPSRC grant EP/V026259/1 and the Alan Turing Institute. JF was previously supported by the EPSRC Programme Grant ``DataSig" EP/S026347/1.

\printbibliography
\clearpage

\appendix

\section{Updated \texttt{bridge} and \texttt{final\_interpolation} functions}

\label{appendix:updated_bridge_finalinterp}

\begin{algorithm}[H]
\caption{Updated \bridge, based on theorems \ref{cor:theory:midpoint_wh} and \ref{thm:theory:midpoint_whk}.}\label{alg:appendix:bridge}

\hspace*{2.5mm} \textbf{Input:} $l$ - depth in tree, $s$ - start of the interval, $\bar{Y}_s, \bar{Y}_u, \bar{Y}_{s,u}$ - BM and L\'{e}vy areas, $\widehat{\rho}$ - PRNG seed
\vspace{1mm}
\begin{algorithmic}
\setstretch{1.5}
\If{$Y = (W, H)$}
    \State \textbf{for} $x \in \set{s, u, [s,u]}$ \textbf{let} $\big( W_x, \bar{H}_x \big) \gets \bar{Y}_x$
    \State $\widehat{\rho}_1, \widehat{\rho}_2 \gets \splitseed(\widehat{\rho}, 2)$
    \State $Z \gets \normalof{0, \frac{1}{16} (u-s) \eye_d, \widehat{\rho}_1}$
    \State $N \gets \normalof{0, \frac{1}{12} (u-s) \eye_d, \widehat{\rho}_2}$
    \State $W_{s,t} \gets \frac{1}{2} W_{s,u} + \frac{3}{2 (u-s)} \bar{H}_{s,u} + Z $, \quad  $W_{t,u} \gets \frac{1}{2} W_{s,u} - \frac{3}{2 (u-s)} \bar{H}_{s,u} - Z $
    \State $\bar{H}_{s,t} \gets \frac{1}{8} \bar{H}_{s,u} - \frac{u-s}{4} Z + \frac{u-s}{4} N $, \quad  $\bar{H}_{t,u} \gets \frac{1}{8} \bar{H}_{s,u} - \frac{u-s}{4} Z - \frac{u-s}{4} N $
    \State $W_t \gets W_s + W_{s,t}$
    \State $\bar{H}_t \gets \bar{H}_s + \bar{H}_{s,t} + \frac{1}{2} \left( t W_s - s W_t \right)$
    \State \textbf{for} $x \in \set{t, [s,t], [t,u]}$ \textbf{let} $\bar{Y}_x \gets \big( W_x, \bar{H}_x \big)$
\Else  \Comment{in this case $Y = (W, H, K)$}
    \State \textbf{for} $x \in \set{s, u, [s,u]}$ \textbf{let} $\big( W_x, \bar{H}_x, \bar{K}_x \big) \gets \bar{Y}_x$
    \State $\widehat{\rho}_1, \widehat{\rho}_2, \widehat{\rho}_3 \gets \splitseed(\widehat{\rho}, 3)$
    \State $Z \gets \normalof{0, \frac{1}{16} (u-s) \eye_d, \widehat{\rho}_1}$
    \State $X_1 \gets \normalof{0, \frac{1}{768} (u-s) \eye_d, \widehat{\rho}_2}$
    \State $X_2 \gets \normalof{0, \frac{1}{2880} (u-s) \eye_d, \widehat{\rho}_3}$
    \State $W_{s,t} \gets \frac{1}{2} W_{s,u} + \frac{3}{2 (u-s)} \bar{H}_{s,u} + Z$, \quad $W_{t,u} \gets \frac{1}{2} W_{s,u} - \frac{3}{2 (u-s)} \bar{H}_{s,u} - Z$
    \State $\bar{H}_{s,t} \gets \frac{1}{8} \bar{H}_{s,u} + \frac{15}{8 (u-s)} \bar{K}_{s,u} - \frac{u-s}{4} Z + \frac{u-s}{2} X_1$, \quad  $\bar{H}_{t,u} \gets \frac{1}{8} \bar{H}_{s,u} - \frac{15}{8 (u-s)} \bar{K}_{s,u} - \frac{u-s}{4} Z - \frac{u-s}{2} X_1$
    \State $\bar{K}_{s,t} \gets \frac{1}{32} \bar{K}_{s,u} - \frac{(u-s)^2}{8} X_1 + \frac{(u-s)^2}{4} X_2$, \quad  $\bar{K}_{t,u} \gets \frac{1}{32} \bar{K}_{s,u} - \frac{(u-s)^2}{8} X_1 - \frac{(u-s)^2}{4} X_2$
    \State $W_t \gets W_s + W_{s,t}$
    \State $B_s^{0,t} \gets W_s - \frac{s}{t} W_t$
    \State $\bar{H}_t \gets \bar{H}_s + \bar{H}_{s,t} + \frac{t}{2} B_s^{0,t}$
    \State $\bar{K}_t \gets \bar{K}_s + \bar{K}_{s,t} + \frac{t-s}{2} \bar{H}_s - \frac{s}{2} \bar{H}_{s,t} + \frac{(t-s)^2 - s^2}{12} B_s^{0,t}$
    \State \textbf{for} $x \in \set{t, [s,t], [t,u]}$ \textbf{let} $\bar{Y}_x \gets \big( W_x, \bar{H}_x, \bar{K}_x \big)$
\EndIf

\State \Return $\bar{Y}_t, \; \bar{Y}_{s,t}, \; \bar{Y}_{t,u}$
\end{algorithmic}
\end{algorithm}

\begin{algorithm}[H]
\caption{Updated \finalinterp, based on theorems \ref{thm:theory:arbitrary_wh} and \ref{thm:theory:arbitrary_whk}.}\label{alg:appendix:final_interp}

\hspace*{2.5mm} \textbf{Input:} $s$ - start time, $r$ - target time, $u$ - end time, $\bar{Y}_s, \bar{Y}_u, \bar{Y}_{s,u}$ - BM and L\'{e}vy areas, $\widehat{\rho}$ - PRNG seed
\vspace{1mm}
\begin{algorithmic}
\setstretch{1.5}
\If{$Y$ = $(W, H)$}
    \State \textbf{for} $x \in \set{s, u, [s,u]}$ \textbf{let} $\big( W_x, \bar{H}_x \big) \gets \bar{Y}_x$
    \State $\widehat{\rho}_1, \widehat{\rho}_2 \gets \splitseed(\widehat{\rho}, 2)$
    \State $X_1 \gets \normalof{0, \eye_d, \widehat{\rho}_1}$
    \State $X_2 \gets \normalof{0, \eye_d, \widehat{\rho}_2}$
    \State $a \gets \frac{\left(r-s\right)^{\frac{7}{2}}\left(u-r\right)^{\frac{1}{2}}}{2(u-s)\sqrt{\left(r-s\right)^{3}+\left(u-r\right)^{3}}} $
    \State $b \gets \frac{\left(r-s\right)^{\frac{1}{2}}\left(u-r\right)^{\frac{7}{2}}}{2(u-s)\sqrt{\left(r-s\right)^{3}+\left(u-r\right)^{3}}} $
    \State $c \gets \frac{\sqrt{3}\left(r-s\right)^{\frac{3}{2}}\left(u-r\right)^{\frac{3}{2}}}{6\sqrt{\left(r-s\right)^{3} + \left(u-r\right)^{3}}} $
    \State $W_{s,r} \gets \frac{r-s}{u-s} W_{s,u} + 6 \frac{(r-s)(u-r)}{(u-r)^3} \bar{H}_{s,u} + 2 \frac{a+b}{u-s} X_1$
    \State $\bar{H}_{s,r} \gets \left( \frac{r-s}{u-s} \right)^3 \bar{H}_{s,u} - a X_1 + c X_2$
    \State $W_r \gets W_s + W_{s,r}$
    \State $\bar{H}_r \gets \bar{H}_s + \bar{H}_{s,r} + \frac{1}{2} \left( r W_s - s W_r \right)$
    \State $\bar{Y}_r \gets \big( W_r, \bar{H}_r \big)$
\Else   \Comment{in this case $Y = (W, H, K)$}
    \State \textbf{for} $x \in \set{s, u, [s,u]}$ \textbf{let} $\big( W_x, \bar{H}_x, \bar{K}_x \big) \gets \bar{Y}_x$
    \State $H_{s,u} \gets \frac{1}{u-s} \bar{H}_{s,u}$ \quad $K_{s,u} \gets \frac{1}{(u-s)^2} \bar{K}_{s,u}$
    \State $\tilde{W}_{s,r} \gets \frac{r-s}{u-s} W_{s,u} + 6 \frac{(r-s)(u-r)}{(u-s)^2} H_{s,u} + 120 \frac{(r-s)(u-r)(\frac{u+s}{2} - r)}{(u-s)^3} K_{s,u}$
    \State $\tilde{H}_{s,r} \gets \frac{(r-s)^2}{(u-s)^2} H_{s,u} + 30 \frac{(r-s)^2(u-r)}{(u-s)^3} K_{s,u}$
    \State $\tilde{K}_{s,r} \gets \frac{(r-s)^3}{(u-s)^3} K_{s,u}$
    \State $\big( \hat{W}_{s,r}, \hat{H}_{s,r}, \hat{K}_{s,r} \big) \gets \normalof{0, \boldsymbol{\Sigma}_r, \widehat{\rho}}$ \quad where $\boldsymbol{\Sigma}_r$ is as in equation \ref{eq:theory:general_times:whk:sigma}
    \State $W_{s,r} \gets \tilde{W}_{s,r} + \hat{W}_{s,r}$
    \State $\bar{H}_{s,r} \gets (r-s) \big( \tilde{H}_{s,r} + \hat{H}_{s,r} \big)$
    \State $\bar{K}_{s,r} \gets (r-s)^2 \big( \tilde{K}_{s,r} + \hat{K}_{s,r} \big)$
    \State $W_r \gets W_s + W_{s,r}$
    \State $B_s^{0,r} \gets W_s - \frac{s}{r} W_r$
    \State $\bar{H}_r \gets \bar{H}_s + \bar{H}_{s,r} + \frac{r}{2} B_s^{0,r}$
    \State $\bar{K}_r \gets \bar{K}_s + \bar{K}_{s,r} + \frac{r-s}{2} \bar{H}_s - \frac{s}{2} \bar{H}_{s,r} + \frac{(r-s)^2 - s^2}{12} B_s^{0,r}$
    \State $\bar{Y}_r \gets \big( W_r, \bar{H}_r, \bar{K}_r \big)$
\EndIf

\State \Return $\bar{Y}_r$
\end{algorithmic}
\end{algorithm}

\section{Proofs of results from Section \ref{sec:theory}}
\label{sec:proofs}

\subsection{Proof of Chen's relation for (W, H, K)}
\label{sec:proofs:chen}

\chenrel*

\begin{figure}[h]
\begin{center}
\includegraphics[width=0.7\textwidth]{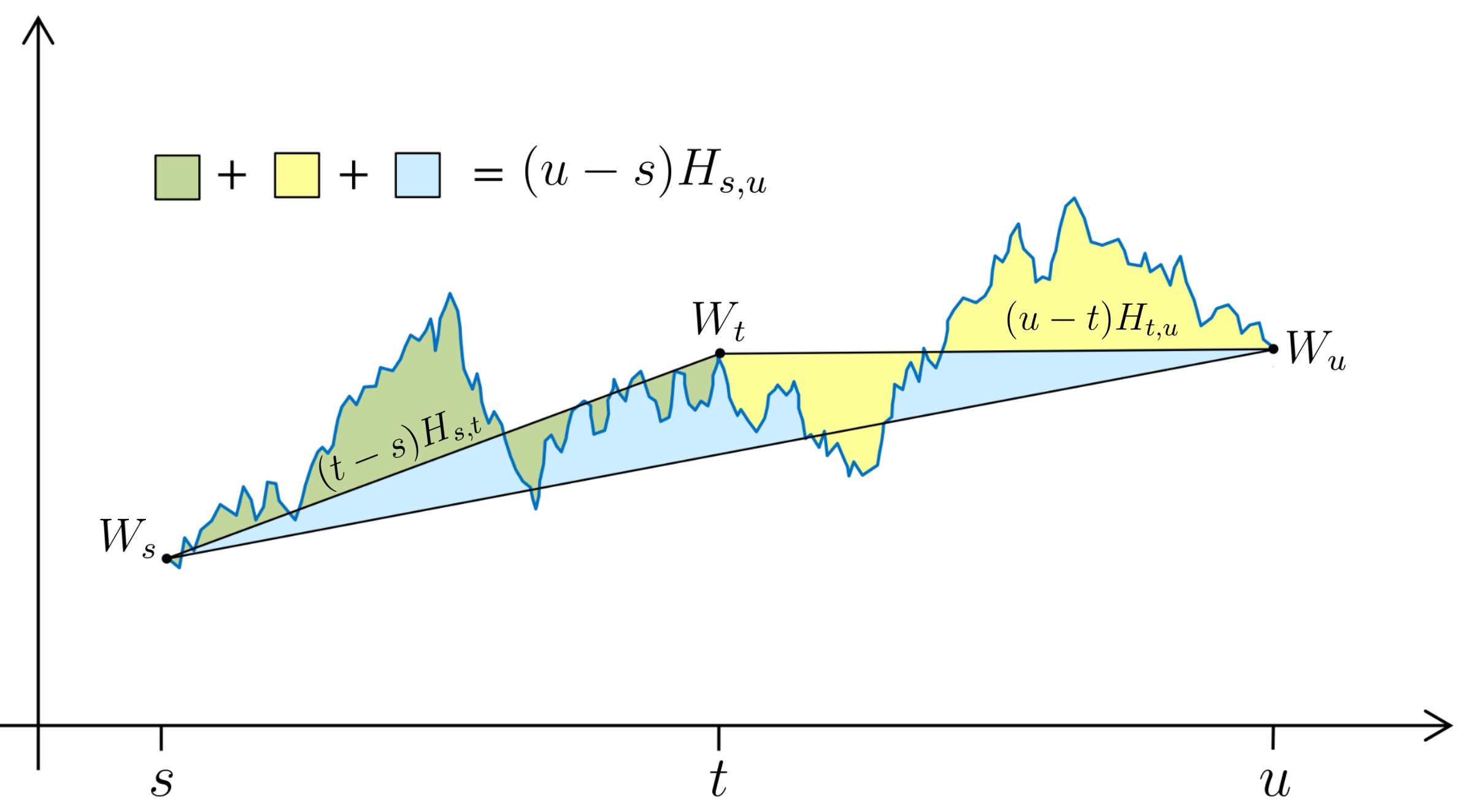}
\caption{Chen's relation for space-time L\'{e}vy area (from \cite{foster2020a}).}
\label{fig:wh_midpoint}
\end{center}
\end{figure}

\begin{proof}
    The relation for $\barH$ appears in \cite{foster2020a} and can be easily deduced using figure \ref{fig:wh_midpoint}. We will derive the identity for $\barK$, which is new here.
    
    Let $0 \leq s < t < u$ be some times. Recall \cref{eq:theory:arbitrary_whk:bridge_bridge}:
    \begin{align*}
    \begin{split}
        B_r^{s,t} &= B_r^{s,u} - \frac{r-s}{t-s} B_t^{s,u} \quad \text{for } \; r \in [s,t] \\
        B_r^{t,u} &= B_r^{s,u} - \frac{u-r}{u-t} B_t^{s,u} \quad \text{for } \; r \in [t,u].
    \end{split}
    \end{align*}

    Substituting this into the definition of $\bar{K}_{s,u} = \int_s^u B_r^{s,u} \left( \frac{u+s}{2} - r \right) \, dr$ gives
    \begin{align*}
        \barK_{s,t} + \barK_{t,u} &= \int_s^t B_r^{s,t} \left( \frac{t+s}{2} - r \right) \, dr + \int_s^t B_r^{t,u} \left( \frac{u+t}{2} - r \right) \, dr \\
        &= \int_s^u B_r^{s,u} \left( \frac{u+s}{2} - r \right) \, dr - \int_s^t B_r^{s,t} \frac{u-t}{2} \, dr + \int_s^t \left( - \frac{r-s}{t-s} B_t^{s,u} \right) \left( \frac{u+s}{2} - r \right) \, dr \\
        & \quad + \int_t^u B_r^{t,u} \frac{t-s}{2} \, dr - \int_t^u \left( - \frac{u-r}{u-t} B_t^{s,u} \right) \left( \frac{u+s}{2} - r \right) \, dr \\
        &= \barK_{s,u} - \frac{u-t}{2} (t-s) H_{s,t} - \frac{1}{12}(t-s)(s + 3u - 4t) B_t^{s,u} \\
        & \quad + \frac{t-s}{2} (u-t) H_{t,u} - \frac{1}{12}(u-t)(3s + u - 4t) B_t^{s,u} \\
        &= \barK_{s,u} - \frac{u-t}{2} \barH_{s,t} + \frac{t-s}{2} \barH_{t,u} - \frac{(u-t)^2 - (t-s)^2}{12} B_t^{s,u}
    \end{align*}
    as required.
\end{proof}

\subsection{Proof of the Midpoint update rule for (W, H, K)}
\label{pf:midpoint_whk}

\midpointWHK*

\begin{proof}
We will prove this as a special case of \cref{thm:theory:arbitrary_whk}, in which we set $r = t = \frac{s+u}{2}$ to obtain
\begin{align}
\label{eq:proofs:midpoint_whk:mean}
\begin{split}
     \E[W_{s,t} | Y_{s,u}] &= \frac{1}{2} W_{s,u} + \frac{3}{2} H_{s,u}, \\[1mm]
    \E[H_{s,t} | Y_{s,u}] &= \frac{1}{4} H_{s,u} + \frac{15}{4} K_{s,u}, \\[1mm]
    \E[K_{s,r} | Y_{s,u}] &= \frac{1}{8} K_{s,u},
\end{split}
\quad
\begin{split}
    \vSigma_t  = (u-s) \begin{bmatrix}
        \frac{1}{16} & - \frac{1}{32} & 0 \\[1mm]
        - \frac{1}{32} & \frac{13}{768} & - \frac{1}{1536} \\[1mm]
        0 & - \frac{1}{1536} & \frac{31}{46080}
    \end{bmatrix}.
\end{split}
\end{align}
The Cholesky decomposition $\vSigma_t = L L^T$ gives
\begin{align}
\label{eq:proofs:midpoint_whk:Cholesky}
\begin{split}
    L  = \sqrt{(u-s)} \begin{bmatrix}
        \frac{1}{4} & 0 & 0 \\
        - \frac{1}{8} & \frac{1}{\sqrt{768}} & 0 \\
        0 & - \frac{1}{\sqrt{3072}} & \frac{1}{\sqrt{2880}}
    \end{bmatrix}.
\end{split}
\end{align}

Combining equations \eqref{eq:proofs:midpoint_whk:mean} and \eqref{eq:proofs:midpoint_whk:Cholesky} gives precisely the desired values of $W_{s,t}, H_{s,t}$ and $K_{s,t}$. We then obtain $W_{t,u}, H_{t,u}$ and $K_{t,u}$ via Chen's relation.

\end{proof}

An alternative proof can be found on page 144 of \cite{foster2020a}.

\subsection{Proof of the general Brownian bridge for (W, H, K)}
\label{sec:proofs:arbitrary_whk}

\arbitraryWHK*

\begin{proof}
    Fix $0 \leq s < u$ and let $r \in [s,u]$.
    Since the vectors $Y_{s,r}$ and $Y_{s,u}$ are jointly Gaussian, we can decompose $Y_{s,r}$ into
    \begin{equation*}
        Y_{s,r} = \widetilde{Y}_{s,r} + \widehat{Y}_{s,r},
    \end{equation*}
    where $\widehat{Y}_{s,r}$ is a zero-mean Gaussian vector independent of $Y_{s,u}$, and 
    $\widetilde{Y}_{s,r} = \E[Y_{s,r} | Y_{s,u}]$
    is a deterministic function of $Y_{s,u}$. To compute $\widetilde{Y}_{s,r}$, we use the Brownian polynomial from 
    \citep{foster2020optimalPoly}:
    \begin{equation}
    \label{eq:theory:brownian_cubic}
        \widetilde{W}_{s,r} \coloneqq \E[W_{s,r} | Y_{s,u}] = \frac{r-s}{u-s} W_{s,u} + 6 \frac{(r-s)(u-r)}{(u-s)^2} H_{s,u} + 120 \frac{(r-s)(u-r)(\frac{u+s}{2} - r)}{(u-s)^3} K_{s,u}.
    \end{equation}
    
    Recall that the Brownian bridge is defined as $B^{s,u}_r \coloneqq W_{s,r} - \frac{r-s}{u-s} W_{s,u}$ with the convention that $B^{s,s}_s \coloneqq 0$. Rearranging gives the following identities:
    \begin{align}
    \label{eq:theory:arbitrary_whk:bridge_bridge}
    \begin{split}
        B_q^{s,r} &= B_q^{s,u} - \frac{q-s}{r-s} B_r^{s,u} \quad \text{for } \; q \in [s,r], \\
        B_q^{r,u} &= B_q^{s,u} - \frac{u-q}{u-r} B_r^{s,u} \quad \text{for } \; q \in [r,u].
    \end{split}
    \end{align}
    
    We can use \cref{eq:theory:brownian_cubic} to compute
    \begin{align*}
        \widetilde{B}_r^{s,u} &\coloneqq \E[B_r^{s,u} | Y_{s,u}] = \widetilde{W}_{s,r} - \frac{r-s}{u-s} \widetilde{W}_{s,u} \\
        &= \frac{(r-s)(u-r)}{(u-s)^2} \left( 6 H_{s,u} + 60 K_{s,u} \right) - 120 \frac{(r-s)^2(u-r)}{(u-s)^3} K_{s,u},
    \end{align*}
    which, when combined with \cref{eq:theory:arbitrary_whk:bridge_bridge}, gives
    \begin{align*}
        \widetilde{B}_q^{s,r} &\coloneqq \E[B_q^{s,r} | Y_{s,u}] \\[1mm]
        &= \frac{(q-s)(r-q)}{(u-s)^2} \left( 6 H_{s,u} + 60 K_{s,u} \right) + 120 \frac{(q-s) \left( (u-r)(r-s) - (u-q)(q-s) \right)}{(u-s)^3} K_{s,u}. 
    \end{align*}
    
    Substituting this into the definition of $H_{s,r}$ and $K_{s,r}$, we obtain a formula for $\widetilde{Y}_{s,r}$:
    \begin{align}
    \label{eq:theory:tildeY_{s,r}}
    \begin{split}
        \widetilde{W}_{s,r} &= \frac{r-s}{u-s} W_{s,u} + 6 \frac{(r-s)(u-r)}{(u-s)^2} H_{s,u} + 120 \frac{(r-s)(u-r)(\frac{u+s}{2} - r)}{(u-s)^3} K_{s,u}, \\[3mm]
        \widetilde{H}_{s,r} &\coloneqq \E[H_{s,r} | Y_{s,u}] = \frac{1}{r-s} \int_s^r \widetilde{B}_q^{s,r} \, dq \\[1mm]
        &= \frac{(r-s)^2}{(u-s)^2} H_{s,u} + 30 \frac{(r-s)^2(u-r)}{(u-s)^3} K_{s,u}, \\[3mm]
        \widetilde{K}_{s,r} &\coloneqq \E[K_{s,r} | Y_{s,u}] = \frac{1}{(r-s)^2} \int_s^r \widetilde{B}_q^{s,r} \left( \frac{r+s}{2} - q \right) \, dq \\[1mm]
        &= \frac{(r-s)^3}{(u-s)^3} K_{s,u}.
    \end{split}
    \end{align}
    
    All that remains is to compute the covariance matrix of $\widehat{Y}_{s,r}$, which depends only on $s, r$, and $u$.
    Since $\left( Y_{s,r} \right)_{r \geq s}$ is a Gaussian process, we have
    \begin{equation}
        \vSigma_r \coloneqq \E[\widehat{Y}_{s,r} \widehat{Y}_{s,r}^T] = \E[Y_{s,r} Y_{s,r}^T] - \E[ \widetilde{Y}_{s,r} \widetilde{Y}_{s,r}^T ].
    \end{equation}
    
    By construction $W_{s,r}, H_{s,r},$ and $K_{s,r}$ are independent (see \cite{foster2020a} for details). Hence $\E[Y_{s,r} Y_{s,r}^T]$ is diagonal with the following entries
    \begin{equation}
        \var \! \left( W_{s,r} \right) = r-s, \quad \var \! \left( H_{s,r} \right) = \frac{r-s}{12}, \quad \var \! \left( K_{s,r} \right) = \frac{r-s}{720}.
    \end{equation}
    
    To compute the matrix $\E[ \widetilde{Y}_{s,r} \widetilde{Y}_{s,r}^T ]$ we can just multiply each line of \cref{eq:theory:tildeY_{s,r}} with each other and then use the independence of $W_{s,u}$, $H_{s,u}$ and $K_{s,u}$.
    
    Taking the difference $\E[Y_{s,r} Y_{s,r}^T] - \E[ \widetilde{Y}_{s,r} \widetilde{Y}_{s,r}^T ]$ gives the desired coefficients.
\end{proof}

\end{document}